\documentclass[final,12pt]{colt2023} 

\usepackage{tikz}
\usepackage{times}
\usepackage{dsfont}

\usepackage{soul}
\usepackage[framemethod=TikZ]{mdframed}

\newcommand{\one}{\mathds{1}}

\definecolor{amber(sae/ece)}{rgb}{1.0, 0.49, 0.0}

\renewcommand{\hat}{\widehat}
\renewcommand{\tilde}{\widetilde}

\newcommand{\pds}{\mathsf{PDS}^*}
\newcommand{\pd}{\mathsf{PDS}}
\newcommand{\isbm}{\textup{ISBM}}
\newcommand{\N}{\mathcal{N}}

\newcommand{\spac}{\vspace{0pt}}
\usepackage{listings}
\newcommand{\dTV}{d_{\rm TV}}

\newcommand{\PCR}{k\textup{-PC}}
\newcommand\numberthis{\addtocounter{equation}{1}\tag{\theequation}}
\newcommand{\sref}[2]{\hyperref[#2]{#1 \ref*{#2}}}

\newcommand{\pc}{\mathsf{PC}}

\renewcommand{\P}{\mathbb{P}}

\newcommand{\var}{\mathsf{var}}

\newcommand\indep{\protect\mathpalette{\protect\independenT}{\perp}}
\def\independenT#1#2{\mathrel{\rlap{$#1#2$}\mkern2mu{#1#2}}}
\newcommand{\Bern}{{\rm Bern}}
\newcommand{\Binom}{{\rm Binom}}

\newcommand{\val}{\mathsf{val}}

\newcommand{\quadand}{ \quad \text{and}\quad }

\newcommand{\DkS}{\mathsf{DkS}}

\coltauthor{%
 \Name{Guy Bresler} \Email{guy@mit.edu}\\
 \addr 77 Massachusetts Ave., Cambridge, MA, USA\\
 \AND
 \Name{Tianze Jiang} \Email{tjiang@mit.edu}\\
 \addr 77 Massachusetts Ave., Cambridge, MA, USA
}

\title[Detection-Recovery and Detection-Refutation Gaps]{Detection-Recovery and Detection-Refutation Gaps \\ via Reductions from Planted Clique}

\begin{document}

\maketitle

\begin{abstract}%
Planted Dense Subgraph (PDS) problem is a prototypical problem with a computational-statistical gap. It also exhibits an intriguing additional phenomenon: different tasks, such as detection or recovery, appear to have different computational limits. A \emph{detection-recovery gap} for PDS was substantiated in the form of a precise conjecture given by Chen and Xu (2014) (based on the parameter values for which a convexified MLE succeeds), and then shown to hold for low-degree polynomial algorithms by Schramm and Wein (2022) and for MCMC algorithms for Ben Arous et al. (2020). 

In this paper 
we demonstrate that a slight variation of the Planted Clique Hypothesis with \emph{secret leakage} (introduced in Brennan and Bresler (2020)), implies a detection-recovery gap for PDS. 
In the same vein, we also obtain a sharp lower bound for refutation, yielding a detection-refutation gap. 
Our methods build on the framework of Brennan and Bresler (2020) to construct average-case reductions mapping secret leakage Planted Clique to appropriate target problems. 
\end{abstract}


\section{Introduction}

The last decade has witnessed a dramatic shift in our understanding of the fundamental limits of high-dimensional statistics problems. Rather than the \emph{statistical limit} being the most relevant quantity governing the minimum amount of data or signal strength needed to solve a problem, it has emerged that for many problems of central importance there is a distinct and often much larger \emph{computational limit} at which computationally efficient algorithms begin to succeed. \cite{Berthet13} showed how a \emph{statistical-computational gap} for a binary variant of sparse PCA follows via reduction from the planted clique hardness conjecture (\sref{Conjecture}{pcconjecture}), spurring intense research activity (see, e.g., \cite{BB20} and references therein). 

In this paper we investigate how the computational complexity of different tasks, including detection, recovery, and refutation, can vary even for the same statistical model. The phenomena of interest are exemplified by the Planted Dense Subgraph (PDS) problem, defined next.

\paragraph{Planted Dense Subgraph ($\pd$).}

A sample from the distribution $\pd(n, k, p, q)$ is obtained by:
    \begin{enumerate}
        \item Sample $G\sim G(n, q)$ an Erd\H{o}s-Renyi graph with edge density $q$.
        \item Select a subset $S$ of vertices  uniformly among the ${n\choose k}$ subsets of size $k$.
        \item Re-sample edges with both endpoints in $S$ independently, including each with probability $p>q$.
    \end{enumerate}
    The \emph{detection} (or decision) problem is to decide, given a graph $G$, between the two hypotheses
    \begin{equation}
        H_0: G\sim G(n, q)\quadand H_1: G\sim \pd(n, k, p, q).
    \end{equation}
The \emph{recovery} problem is to (exactly) find the planted support $S$. (Weaker notions of recovery can be found in \sref{Section}{Defs}.)

The special case of $\pd$ where $p=1$ is known as the Planted Clique ($\pc$) problem. Let $G(n,k,p)=\pd(n,k,1,p)$. We denote by $\textsc{PC}_D(n, k, p)$ the problem of deciding between
$$H_0: G\sim G(n, p)\quadand H_1: G\sim G(n, k, p)\,.$$ Both detection and recovery have efficient (polynomial-time) algorithms whenever $k=\Omega(\sqrt{n})$ (\cite{alon1998finding}), but
a growing body of evidence (\cite{BHK16,feldman2017statistical}) suggests that these problems become hard for clique size $k=n^\beta$ with $\beta<1/2$.

\begin{conjecture}[PC Conjecture] \label{pcconjecture}
Fix constant $p \in (0, 1)$. Suppose that $\{ A_n \}$ is a sequence of randomized polynomial time algorithms $A_n : G_n \to \{0, 1\}$ and $k_n$ is a sequence of positive integers satisfying that $\limsup_{n \to \infty} \log_n k_n < \frac{1}{2}$. If $G$ is an instance of $\textsc{PC}_D(n, k, p)$, then
$$\liminf_{n \to \infty} \left( \P_{H_0}\left[A_n(G) = 1\right] + \P_{H_1}\left[A_n(G) = 0\right] \right) \ge 1.$$ 
\end{conjecture}
In our work, we will use a (stronger) variation of this assumption proposed by \cite{BB20} where some structural information of the planted clique is assumed (the \emph{secret leakage}). See \sref{Conjecture}{conj: secret leak} and the associated discussion.

\subsection{Computational feasibility of $\pd$} 
\paragraph{$\pd$ Detection.} Feasibility of detection in $\pd$ is described by a \emph{phase diagram} (see Fig.~\ref{fig: summary}) indicating for each possible parameter choice whether the problem is: (1) information-theoretically impossible, (2) solvable in principle but computationally hard, or (3) solvable in polynomial time.
Complete phase diagrams were shown by reduction from $\pc_D$ in the regime $q=\Theta(1)$ by \cite{MaWu}\footnote{In the regime $q=\Theta(1)$, $\pd$ is easily seen to be computationally equivalent up to log factors in the parameter values to the Gaussian matrix model with corresponding means.},
for the sparse regime $p=cq$ for constant $c$ and $q=1/\mathrm{poly}(n)$ by \cite{HWX15}, and by \cite{BBH18} for a general regime interpolating between the two. 
Despite the similarity between $\pc_D$ and $\pd_D$, it is non-trivial to construct reductions that are tight against algorithms, since $\pd_D$ exhibits a trade-off between subgraph size and signal strength.

In all of the above parameter regimes, whenever $k=\omega(\sqrt{n})$ the optimal polynomial-time test $T_\mathrm{sum}$ simply compares the total number of edges to a threshold. A second moment calculation shows that
$$
T_\mathrm{sum}\text{ succeeds w.h.p. if}\quad \frac{k^4(p-q)^2}{n^2q(1-q)}=\omega(1)\,.
$$
By its nature, success of the sum statistic yields no information whatsoever about the location of the planted dense subgraph. What can be said about recovery?

\label{sec: recovhard}
\paragraph{$\pd$ Recovery.}
The best currently-known algorithms (such as spectral, semi-definite programming, and low-degree polynomials) for \emph{recovery} 
turn out to require a dramatically higher signal strength (\cite{CX14,HWX16}). 
The following conjecture posits that this signal strength is optimal for the recovery problem (\cite{CX14,HWX15}). 

\begin{conjecture} [PDS recovery conjecture]\label{conj: rec}
Suppose $G_n\sim \pd(n, k_n, p_n, q_n)$. If $k=\omega(\sqrt{n})$ and
$$\lim\sup \log_n \frac{k^2(p-q)^2}{nq(1-q)}<0,$$
then no polynomial algorithm $\mathcal{A}: G\to {[n] \choose k}$ can achieve exact recovery of $\pd$ asymptotically.
\end{conjecture}

The lower bound in this conjecture has been shown for restricted classes of algorithms: in \cite{SchrammWein} for low-degree polynomials and \cite{BWZ20} for Markov Chain Monte Carlo algorithms. 

\paragraph{Recovery lower bound via reduction?} 
Lower bounds have been shown for a wide variety of detection problems via reduction from $\pc$, and for the majority of these problems recovery is algorithmically feasible in the same parameter regime (to within a constant factor) in which detection is algorithmically feasible. Yet for problems where recovery seems strictly harder than detection, demonstrating a detection-recovery gap via reduction from $\pc$ has remained elusive. Attempts in this direction include those of \cite{CLR17} showing hardness for a matrix model with highly correlated entries (different from the independent edges in $\pds$), and  \cite{BB20} showed that the conjectured recovery lower bound follows from the $\pc$ conjecture for a \emph{semirandom} variant of $\pd$ where an adversary may ``helpfully" remove edges outside of the dense subgraph. The main question motivating our work is: 

\begin{center}\emph{Can a detection-recovery gap be shown for Planted Dense Subgraph via reduction?}
\end{center}

A first conceptual challenge is that, as shown by \cite{Alonred}, detection and recovery for $\pc$ are \emph{equivalent}. What this means is that the detection-recovery gap appearing in $\pd$ is inherent to $\pd$, and in particular, we cannot simply map from $\pc$ detection and $\pc$ recovery separately. 

In fact, our reductions will still map to detection problems (with implications to recovery). But we cannot simply map to the $\pd$ detection hypotheses $\pd_D$: Otherwise, we would be mapping a conjecturally hard instance of $\pc$ to an easy instance of $\pd_D$! Our goal in this paper is considerably more modest than to refute the planted clique conjecture, so we must find another way.

 \subsection{Contributions}
\begin{figure}[t]
    \centering
    {
\begin{tikzpicture}[scale=0.3] 
\tikzstyle{every node}=[font=\scriptsize]
\def\xmin{0}
\def\xmax{11}
\def\ymin{0}
\def\ymax{11}

\draw[->] (\xmin,\ymin) -- (\xmax,\ymin) node[right] {$\beta$};
\draw[->] (\xmin,\ymin) -- (\xmin,\ymax) node[above] {$\alpha$};

\node at (5, 0) [below] {$\frac{1}{2}$};
\node at (6.66, 0) [below] {$\frac{2}{3}$};
\node at (10, 0) [below] {$1$};
\node at (0, 0) [left] {$0$};
\node at (0, 10) [left] {$2$};
\node at (0, 3.33) [left] {$\frac{2}{3}$};
\node at (0, 5) [left] {$1$};

\filldraw[fill=cyan, draw=blue] (0,0) -- (5, 0) -- (6.66, 3.33) -- (0, 0);
\filldraw[fill=amber(sae/ece), draw=red] (5, 0) -- (6.66, 3.33) -- (10, 5) -- (5, 0);
\filldraw[fill=green!25, draw=green] (5, 0) -- (10, 5) -- (10, 0) -- (0, 0) -- (0, 0) -- (5, 0);
\filldraw[fill=red, draw=magenta] (6.66, 3.33) -- (10, 5) -- (10, 10) -- (6.66, 3.33);
\filldraw[fill=gray!25, draw=gray](0, 0) -- (6.66, 3.33) -- (10, 10) -- (0, 10) -- (0, 0);

\node at (4, 1) {HH};
\node at (8.5, 1.2) {EE};
\node at (7.1, 3.0) {\tiny EH};
\node at (8.5, 5.25) {EI};
\node at (4, 5.25) {II};
\end{tikzpicture}}\hspace{30mm}
{
\begin{tikzpicture}[scale=0.3]
\tikzstyle{every node}=[font=\scriptsize]
\def\xmin{0}
\def\xmax{11}
\def\ymin{0}
\def\ymax{11}

\draw[->] (\xmin,\ymin) -- (\xmax,\ymin) node[right] {$\beta$};
\draw[->] (\xmin,\ymin) -- (\xmin,\ymax) node[above] {$\alpha$};

\node at (5, 0) [below] {$\frac{1}{2}$};
\node at (6.66, 0) [below] {$\frac{2}{3}$};
\node at (7.5, 0) [below] {$\frac{3}{4}$};
\node at (10, 0) [below] {$1$};
\node at (0, 0) [left] {$0$};
\node at (0, 10) [left] {$2$};
\node at (0, 3.33) [left] {$\frac{2}{3}$};
\node at (0, 5) [left] {$1$};

\filldraw[fill=cyan, draw=blue] (0,0) -- (5, 0) -- (6.66, 3.33) -- (0, 0);
\filldraw[fill=amber(sae/ece), draw=red] (5, 0) -- (6.66, 3.33) -- (7.5, 3.75) -- (5, 0);
 \filldraw[fill=green!25, draw=green](5, 0) -- (10, 5) -- (10, 0) -- (0, 0) -- (0, 0) -- (5, 0);
\filldraw[fill=red, draw=magenta] (6.66, 3.33) -- (10, 5) -- (10, 10) -- (6.66, 3.33);
\filldraw[fill=gray!25, draw=gray] (0, 0) -- (6.66, 3.33) -- (10, 10) -- (0, 10) -- (0, 0);

\node at (4, 1) {HH};
\node at (8.5, 1.2) {EE};

\node at (8.5, 5.25) {EI};
\node at (4, 5.25) {II};
\end{tikzpicture}}
\vspace{-3mm}
    \caption{\footnotesize The pictures above (left: Detection vs Refutation; right: Detection vs Recovery) concerns $\pd(n, k, p, q)$ when $p, q$ are bounded away from 0 and 1, and $k\in\tilde\Theta(n^{\beta}), D_{KL}(p\|q)\in\tilde\Theta(n^{-\alpha})$, where E denotes \underline{e}asy, H (computationally) \underline{h}ard, and I (statistically) \underline{i}ntractable, hence the orange \textcolor{amber(sae/ece)}{EH} (computationally easy to detect but hard to refute/recover) is our main results. Our statistical \textcolor{red}{EI} and computational \textcolor{amber(sae/ece)}{EH}  characterization of refutation (left) in this density regime are both novel. The orange-white region in the right denotes the conjectural \textcolor{amber(sae/ece)}{EH} regime, where we close for orange and leave white open.\vspace{-5mm}}
    \label{fig: summary}
\end{figure}
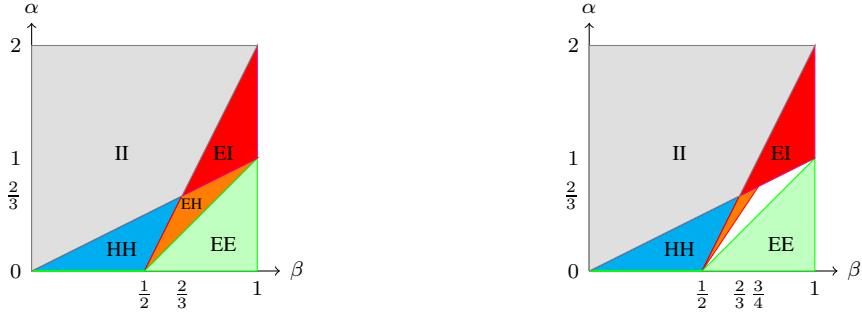
In this work, we will utilize the insight that by constructing different statistical models with similar underlying properties, tailored to corresponding inference tasks, we can go beyond the simple detection boundary to prove tighter results. Our main 
contributions are:
\begin{itemize}
    \item 
    We present the first reduction-based evidence of a computational \underline{\emph{detection-recovery gap}} (\sref{Corollary}{RECOVBOUND}) for recovering the hidden community in planted dense subgraph, via an average-case reduction from Planted Clique with secret leakage (\sref{Conjecture}{conj: secret leak}). 
    \item We show how detection hardness for the two-community Imbalanced Stochastic Block Model (ISBM), shown by reduction from \sref{Conjecture}{conj: secret leak} by \cite{BB20}, can be used to obtain a log-optimal lower bound on refuting dense $k$-subgraphs in $G(n, p)$ and Gaussian principal submatrix with large mean. This matches the algorithm-specific results of \cite{BHK16} and \cite{SoSrefute} and shows a reduction-based \underline{\emph{detection-refutation gap}}.
    \item Combining our results with existing reductions yields analogous results also for other average-case planted models such as Gaussian biclustering and biased sparse PCA. This yields detection-recovery gaps for these problems and answers a question from \cite{BBH18}. 
    \item Finally, we also give insight into the relationships between the statistical boundaries for the above problems, including showing a nearly sharp limit on refuting densest $k$-subgraphs in Erd\H{o}s R\'enyi Graphs via a novel reduction from recovery.
\end{itemize}


\subsection{Reductions and Other Evidence for Hardness}

\paragraph{Average-case Reductions.} We will define an (average-case) reduction (in total variation) from two source distributions $P_0, P_1$ to target pair $Q_0,Q_1$ as a (random) polynomial-time computable map $\Phi$ such that the pushforward $\dTV(\Phi(P_i), Q_i)=o(1)$ for $i=1,2$. The implication is that if $P_0, P_1$ are computationally hard to distinguish, then the same holds true for $Q_0$ versus $Q_1$: any poly-time algorithm $\mathcal{A}$ for the latter task would yield a poly-time algorithm $\mathcal{A} \circ \Phi$ for the former by composing with the reduction, contradicting the presumed hardness of $P_0$ versus $ P_1$. 

While reductions form the bread and butter of complexity theory, there is a general sentiment in the community that average-case reductions are notoriously delicate. Such reductions must not only map to a valid problem instance, they must precisely map entire probability distributions. 

The upside is that reductions can give the strongest possible evidence for computational hardness, and moreover, they demonstrate a connection between two formerly disparate problems which is often of interest independent of hardness. We refer to \cite{BB20} for a review of the reductions literature. 



\paragraph{Algorithm-specific hardness} There have been numerous results showing lower bounds for classes of algorithms and we mention a few of the results that relate to $\pd$. In \cite{BHK16}, a lower bound for refutation of large cliques 
in $G(n, \frac{1}{2})$ was shown for the Sum-of-Squares hierarchy. \cite{SchrammWein, 1v2Hard} sharply characterize the power of low-degree polynomials for recovery in $\pd$. 
The overlap gap framework, introduced in \cite{gamarnik2014limits}, connects algorithmic infeasibility with properties of the solution space geometry (see also \cite{OGP,gamarnik2019landscape}). Other relevant results include that of \cite{feldman2017statistical} on the statistical query model analyzed in the case of a bipartite ``samples" version of $\pc$ and \cite{brennan2021statistical} relating the power of low-degree polynomials and the statistical query model.




\subsection{Inference tasks beyond decision}
\label{Defs}
 Denote by $H_0$ the null hypothesis (usually an Erd\H{o}s-R\'enyi Graph), and $H_1$ is a graph with a planted structure with support $v\in\{0,1\}^n$. Consider a valuation function $\val$ on graphs such that: $\P_{H_0}(\val(G)<\delta-\epsilon)$ and $\P_{H_1}(\val(G)>\delta+\epsilon)$ are both $1-o_n(1)$. In the case of PDS, $\val$ is the densest-$k$-subgraph density. Consider the following:

\paragraph{Refutation} A refutation algorithm with success probability $p$ is a (randomized) algorithm $\mathcal{A}$ supported on all graphs of size $n$:
\begin{itemize}
    \item If $\val(G)>\delta+\epsilon$, then $\mathcal{A}(G)=1$.
    \item For $G\sim \P_{H_0}(\,\cdot\,|{\val(G)<\delta-\epsilon})$, output $\mathcal{A}(G)=0$ with probability at least $p$.
\end{itemize}

\paragraph{Recovery} Let $\pi$ be a distribution over size $k$ planted supports $v\in\{0,1\}^n$, and for each $v$ let $P_v$ be a distribution over planted graphs. 
Let $G\sim P = \mathbb{E}_{v\sim \pi}P_v$. A recovery blackbox $\mathcal{A}: G\to \{0,1\}^n$ 
is said to achieve 
\begin{enumerate}
\item  \textit{Partial recovery:}\;\;\;\; If $\;\mathbb{E}[ v^T\mathcal{A}(G)]=\Omega(||v||_1)$.
    \item \textit{Weak recovery:}\;\;\;\; If $\;\mathbb{E}[ v^T\mathcal{A}(G)]=\|v\|_1-o(||v||_1)$.
    \item \textit{Exact (precise) recovery:}\;\;\;\; If $\;\P[ \mathcal{A}(G)=v]= \Omega(1)$.
\end{enumerate}
In most models we consider, these variants of recovery only differ in sub-polynomial factors (via reduction in \sref{Appendix}{recovs}). 
We further remark that a refutation algorithm is only evaluated on the input distribution $H_0$, whereas a recovery algorithm is only evaluated on the distribution $H_1$. The latter fact was leveraged by \cite{SchrammWein} and both will be crucial to our proofs.
\begin{lemma}[Informal, see \sref{Lemma}{recovtodetect}] \label{lem: oracle to detection}
For any $\tilde H_0$ that does not have a $k$-subgraph with density above $\frac{p+q}{2}$ with high probability, weak recovery oracles nontrivially distinguish $\tilde H_0, H_1=\pd$.
\end{lemma}

\subsection{Planted Clique and Secret Leakage}
We require a slight modification of the planted clique conjecture, proposed by \cite{BB20}:
Instead of a uniformly located clique, the clique is sampled according to some distribution $\rho$ over the $n\choose k$ possible clique positions. One may interpret this as a form of \emph{secret leakage}, whereby some information about the clique position has been revealed to the algorithm.



The form of secret leakage we will use in our reductions is $k\textsc{-PC}_D(n, k, p)$, where there is some fixed (known) partition $E$ of $[n]$ into $k$ equally-sized subsets, and under $H_1$ the planted set is obtained by selecting exactly one node uniformly from each part. We refer to the corresponding hardness assumption as the $k$-PC conjecture.

\begin{conjecture}[$k$-PC Conjecture] \label{conj: secret leak}
Fix constant $p \in (0, 1)$. Suppose that $\{ A_n \}$ is a sequence of randomized polynomial time algorithms $A_n : G_n \to \{0, 1\}$ and $k_n$ is a sequence of positive integers satisfying that $\limsup_{n \to \infty} \log_n k_n < \frac{1}{2}$. Then if $G$ is an instance of $k\textsc{-PC}_D(n, k, p)$, it holds that
$$\liminf_{n \to \infty} \left( \P_{H_0}\left[A_n(G) = 1\right] + \P_{H_1}\left[A_n(G) = 0\right] \right) \ge 1.$$ 
\end{conjecture}
We refer to \cite{BB20} for a general leakage PC conjecture and supporting evidence. When the amount of leaked information is small enough, 
both low-degree polynomials and statistical query algorithms succeed only above the same $\sqrt{n}$ clique size as in ordinary PC. 

\begin{remark}[Binomial planted set]\label{rmk:1} In the literature it is sometimes assumed that the planted set is of fixed size $k$, and other times it is of binomial size (where each node is planted with probability $k/n$ independently). We use a fixed size $k$ and note that all of our (hardness) results extend to corresponding binomial versions by virtue of closeness of the hypergeometric and binomial distributions in appropriate parameter regimes (which can be understood as an instance of a finite de Finetti type theorem \cite{diaconis1980finite}). In particular, one may carry out a reduction by keeping a random $o(n)$ sized fraction of the nodes and discarding the rest. 
\end{remark}

\section{Reduction Techniques Overview}
\subsection{Selecting hypotheses}
    As discussed in \sref{Section}{sec: recovhard}, we cannot map to the standard two $\pd$ hypotheses. A key insight from \sref{Section}{Defs} is that while detection concerns both $H_0$ and $H_1$, all other tasks deal with only one of the two hypotheses. Specifically, for any pair of hypotheses with distributions satisfying the $\val$ criteria, recovery algorithms are only evaluated on an input distributed according to  $H_1$ and not $H_0$. To this end, we are free to select qualifying ``quiet" hypothesis $\tilde H_0$ that is not Erd\H{o}s-Renyi such that it has a harder decision task and imply stronger recovery lower bounds. Similarly, for refutation we may map to $\tilde H_1$ that is different from the standard $H_1$.
    
 Now, suppose that we want to map from the two hypotheses in $\pc$ to $\tilde H_0, H_1$ in a target graph such that $H_1$ is $\pd$ (so that a recovery blackbox enables us to test between $\tilde H_0$ and $H_1$). We have the following naturally competing constraints:
 \begin{enumerate}
     \item For a recovery blackbox to achieve detection, $\val(G)|_{H_0}$ has to be small with high probability, suggesting the fact that $\tilde H_0$ has to be \emph{far} from $H_1$, with respect to some metric.
     \item We need to construct a reduction. From a data-processing inequality perspective, this means that $\tilde H_0$ has to be \emph{closer} to $H_1$ than the distance between source hypotheses.
 \end{enumerate}
It turns out that for recovery, the correct $\tilde H_0$ is extremely hard to find (Appendix B in \cite{SchrammWein}), and even for good $\tilde H_0$ candidates, constructing a tight reduction seems challenging.
However, we will show that by changing $H_0$ to simply match the first-moment in $H_1$, one can achieve a $\tilde H_0$ realizing a non-trivial gap between detection and recovery in $\pd$, while still being feasible for us to map to from $\pc$. Changing $H_0$ as we do here was also analyzed for the case of low-degree polynomials by \cite{SchrammWein}. 
Note that \cite{BB20} in their result on semirandom $\pd$ modified $H_1$, rather than $H_0$, and we will use this same reduction to demonstrate a detection-refutation gap. 
We define the following models:

\paragraph{Mean-corrected Planted Dense Subgraph ($\pds$).}
Consider $\pd$ with $H_0$ modified to prevent success of the obvious first moment test. Consider edge strengths $q<p_0<p$ and size $k$ such that
$$p_0=q+\gamma=p-\Big(\frac{n^2}{k^2}-1\Big)\gamma$$ 
and define $\pds(n,k,p,q)$ as hypothesis testing between
\begin{equation}\label{PDSSTAR}
    H_0: G\sim G(n, p_0)
   \quadand
    H_1: G\sim \pd(n, k, p, q).
\end{equation}

\paragraph{Imbalanced Stochastic Block Model (ISBM).}
Consider a two-community Stochastic Block Model $\isbm(n, k, P_{11}, P_{12}, P_{22})$ to be the graph model generated by sampling $S_1\sim {[n]\choose k}$ and $S_2=[n]\setminus S_1$. Connect nodes $u\in S_i, v\in S_j$ with probability $P_{ij}=P_{ji}$. Moreover, we force the degree constraints \textit{on each node}
$$n\cdot P_0=k\cdot P_{11}+(n-k)\cdot P_{12}=k\cdot P_{12}+(n-k)\cdot P_{22}$$ and formulate the decision problem $\isbm_D$ as (let $r=n/k$):
\begin{equation}\label{ISBM}
    H_0: G\sim G(n, P_0),\;\;\;\;H_1: G\sim \isbm(n, r, P_{11}, P_{12}, P_{22}).
\end{equation}
This model can be considered as a mean-field analogue of recovering a first community
in a general balanced $r$-block SBM model (keeping one block while averaging out the rest).
\footnote{
Note that both models contain a dense subgraph (high $\val$), and $\pds$ is just a translated $\pd$.}

\subsection{Signal transformation}
We start our reduction by viewing our problem as a \emph{planted bits} problem, which is simply a vector $v\sim\Bern(q)^{\otimes n}$ with planted bits $v_I\sim\Bern(p)$ at the index set $I\subseteq [n]$ with a different bias. 
Concretely, because of the one clique vertex per partition assumption of $k$-$\pc$, each $\frac nk \times \frac nk$ block of the adjacency matrix has \emph{a single} planted 1 entry.
All of the reductions we consider can be viewed as mapping a set of planted bits to another desired target set of planted bits with a larger planted size and specific biases.

The difficulty at the core is thus the following: 
\emph{how to transform the planted bits distribution with unknown location to a desired target distribution while not losing signal-to-noise ratio} (measured by the KL-divergence) between planted and null bits and the size of planted location $I$ (\cite{BBH19}), so that the target instance remains at the threshold of algorithmic feasibility. 

As in \cite{BB20}, we will use Gaussian distributions as intermediate steps in transforming from $k$-PC. While Bernoulli data are challenging to non-trivially transform without signal loss, we will leverage the nice behavior of Gaussians under linear maps, enabling us to carefully control the added noise within the transformation (as discussed in the next subsection). 

To see the approximate equivalence between Gaussians and Bernoulli variables, we note that a Gaussian $\N(\mu, 1)$ can be readily mapped to $\Bern(\Phi(\mu))$, where $\Phi$ is the Gaussian CDF, by thresholding at 0. If $\mu\ll 1$, the KL-divergence decreases only by a numerical constant factor independent of $\mu$. In the other direction, a rejection sampling procedure can map a pair of Bernoulli variables to a pair of Gaussians with little information loss\footnote{This process introduces a log-factor, which is the (only) reason in later sections we ignore poly-log factors in rates.}:

\begin{lemma}[Gaussian Rejection Kernels -- \cite{MaWu,BBH18}] \label{lem:5c}
Let $R$ be a parameter and suppose that $0 < q < p \le 1$, $\min(q, 1 - q, p-q) = \Omega(1)$. Suppose that $\mu< \big(1\land \frac{\delta}{2 \sqrt{6\log R + 2\log (p-q)^{-1}}}\big)$ where $\delta = \min \left\{ \log \left( \frac{p}{q} \right), \log \left( \frac{1 - q}{1 - p} \right) \right\}$ , then there exist map $\textsc{rk}(\cdot)$ can be computed in $\text{poly}(R)$ time such that the push-forward maps satisfy
$$\dTV\big(\textsc{rk}( \textnormal{Bern}(p)),  \mathcal{N}(\mu, 1) \big) = O\left(R^{-3}\right) \quadand
\dTV\big(\textsc{rk}( \textnormal{Bern}(q)), \mathcal{N}(0, 1) \big) = O\left(R^{-3}\right).$$
\end{lemma}

Now that we have a Gaussian signal with planted mean, we apply a \emph{rotation} (treating the entire matrix as a vector). Specifically, in \cite{BB20} the following process \textsc{Bern-Rotations} was introduced to transform an instance of $\N(v, I_\ell)$ where $v\in\mathbb{R}^{\ell}$ contains signal.\begin{enumerate}
    \item We right-multiply the Gaussian vector by a \emph{design matrix} $A\in\mathbb{R}^{\ell\times m}$, which yields $vA+\N(0, AA^T)$. Denote the square of the top-singular value of $A$ to be $\lambda=\sigma^2(A)$.
    \item On the re-scaled result vector $\N({\lambda^{-1/2}}vA, AA^T/\lambda)$, we can add a Gaussian noise $\N(0, I-AA^T/\lambda)$ independent of $\mu$ to get exactly $N(\frac{vA}{\sigma(A)}, I_n)$, which has unit variance.
\end{enumerate}
 In short, \emph{we transform signals as mean vectors of isotropic Gaussian distributions by rotating the space and paying an extra whitening noise to produce an isotropic distribution again}.

\begin{lemma}[Dense Bernoulli Rotations -- Lemma 8.1 in \cite{BB20}] \label{lem:bern-rotations}
Let $m$ and $\ell$ be positive integers and let $A \in \mathbb{R}^{\ell \times m}$ be a matrix with singular values all at most $\lambda > 0$. Let $R$, $0 < q < p \le 1$ and $\mu$ be as in Lemma \ref{lem:5c}. Let $\mathcal{A}$ denote $\textsc{Bern-Rotations}$ applied with rejection kernel parameter $R$, Bernoulli biases $0 < q < p \le 1$, output dimension $m$, matrix $A$ with singular value upper bound $\lambda$ and mean parameter $\mu$. Then $\mathcal{A}$ runs in $\textnormal{poly}(\ell, R)$ time and
\begin{align*}
\dTV\left( \mathcal{A}\left( \textup{PB}(\ell, i, p, q) \right), \, \N\left( \mu \lambda^{-1} \cdot A_{i}, I_m\right) \right) &= O\left(\ell\cdot R^{-3}\right) \\
\dTV\left( \mathcal{A}\left( \textnormal{Bern}(q)^{\otimes \ell} \right), \, \N\left( 0, I_m\right) \right) &= O\left(\ell\cdot R^{-3}\right)
\end{align*}
for all $i \in [\ell]$, where $A_i$ is the $i$th row of $A$ and $\textup{PB}(\ell, i, p, q)$ is the distribution on $\{0, 1\}^{\otimes \ell}$ where the $i$th bit is sampled from $\Bern(p)$ and all others from $\Bern(q)$ independently.
\end{lemma}

As noted earlier, with the $\PCR$ constraint we have $r^2$ different blocks, given by the partition, where each block has exactly one planted bit. This allows us to view the entire $\PCR$ matrix as a collection of PB problems and apply \textsc{Bern-rotations} on each $(n/k) \times (n/k) $ matrix ($\ell=n^2/k^2$).

There are two remaining things to consider. Firstly, how to get from Gaussians back to Bernoullis and the final output, and secondly, what criteria does our design matrix $A\in\mathbb{R}^{k^2\times k^2}$ have to follow. For the first step, as noted above, transforming $\N(0, 1), \N(\nu, 1)$ to two Bernoullis by thresholding at 0 will not lose too much information measured by $\dTV$ when $\mu$ is small, and the transformed signal will be approximately $\Bern(0.5)$ and $\Bern(0.5+\frac{\mu}{\sqrt{2\pi}})$ since the Normal CDF is continuous. 

To deal with the other part, we need each row of $A$ to \emph{map directly to the edge density parameter} of output. Specifically, for any (unknown) input PB instance, it gets mapped to an unknown row of $A$, which then becomes the output $\pd$ mean. Our \emph{design} in $A$ is thus formulated as: how to find a suitable $A$ such that each row of $A$ corresponds to a possible mean adjacency matrix in target $\pd$.

\subsection{Design matrices}
We first remark that the key factor in \textsc{Bern-Rotations} is the added noise $\N(0, I-AA^T/\lambda)$, which will in fact be the only part of our reduction process that may introduce irreversible signal loss. Consequently, we want to construct matrix $A$ such that $I-AA^T/\lambda$ is as small as possible: \emph{$A$ has to be close to an isometry.} Let us first assume that $\sigma(A)=1$ for simplicity.

As an example, suppose one wants to map from $\PCR$ to the Gaussian version of $\pd$ (i.e., $\N(\gamma, 1)^{\otimes k\times k}$ planted in $\N(0, 1)^{\otimes n\times n}$) with tight recovery boundary  such that $k^2\gamma^2\sim n$. As $\PCR$ contains at most ${n}$ planted bits yet the squared $\ell_2$ norm of the target mean matrix is exactly $k^2\gamma^2=\Omega (n)$, the sum of squared $\ell_2$ norms of the $n$ column vectors $A_i$ being mapped to should be at least $\Omega(n)$, which (informally) implies that the design matrix $A$ has to be an almost perfect isometry given $\sigma(A)=1$. 


 Having independently generated random columns would allow to apply random matrix spectral bounds. For example, a matrix with i.i.d entries from some fixed distribution was used by \cite{BB20}. They proved that this methods achieves the desired spectral bound, but each column has a random number of planted bits resulting in binomial planted size rather than the desired fixed size (c.f., \sref{Remark}{rmk:1}).

Viewing the design matrix structure as the adjacency matrix of some graph, where i.i.d. matrices corresponds to Erd\H{o}s-R\'enyi graphs, a natural alternative is regular graphs. These satisfy our fixed size constraint. Moreover, considering the tight recovery reduction again, one also needs all rows to have squared norms of $O(1)$ while summing up to $\Omega(n)$, making it an implicit regularization in our construction that all row norms have $\Theta(1)$ norm. This fact provides a crucial motivation into directed \emph{regular graph} models for generating matrices such that the row norms and column norms align, and the columns are roughly independent (i.e. perpendicular). 

\subsection{Singular value from recentering}
We will now focus on what happens in each sub-block with size $m=n/k$ given by partitioning $\PCR$, and treat it as our main target.\footnote{With a slight abuse of notations, we note this is different from the target planted size in $\pd$.}
A line of works (\cite{RegGraphSpectral,LLY17}) have given high probability bounds on the spectral norm $\|A-\mathbb{E}(A)\|_{op}$ of adjacency matrix $A$ for a random graph $G$ with given degree distributions (planted signal). Here we consider when $A$ is the adjacency matrix of a directed $d$-regular graph (each node has out-degree and in-degree exactly $d$). In this case the operator norm of concentration can be expressed with the second largest singular value of $A$. In \cite{RegGraphSpectral}, a (tight) high probability upper bound on the said quantity has been proven when $m^{\alpha}<d<m/2$ we have
$|s_2(A)|\leq C_{\alpha, m}\sqrt{d}$ with high probability. With this result, we can establish the following lemma that will lead to the ultimate design matrix by taking the (translated) Kronecker product to make it $m^2\times m^2$:
\begin{lemma}[Random matrix with regular constraints]\label{RegularMatrix}
Given constant $\alpha>0$, there exists a constant $C_\alpha$, such that for a $m\times m$ (random) matrix $R=R_{m, 1/r}$ where $r<m^{1-\alpha}$ is an even divisor of $m$, with entries sampled from the following procedure:
\begin{enumerate}
    \item Sample $G$ uniformly from all directed $m/r$-regular graphs with size $m$.
    \item $R_{ij}=\frac{-1}{\sqrt{mr}}+1_{e_{ij} \in E_G}\cdot\sqrt{\frac{r}{m}}$ for $j\neq i$ off diagonal, $R_{ii}=\frac{-1}{\sqrt{mr}}$ on the diagonal.
\end{enumerate}
Then with probability $1-o_m(1)$ this matrix satisfies
$\|R\|_{op}\leq C_{\alpha}.$
\end{lemma}
This (centered) matrix has a nice property in that it is an approximate isometry, where each row has $\frac{r-1}{r}$ fraction of $-\gamma=-{1}/{\sqrt{mr}}$ and $1/r$ fraction of $(r-1)\gamma$ with norm 1. However, it is not yet in the form we target (recall that we want to each column to map to the mean of the $m\times m$ adjacency matrix of a graph). It is natural to view the target $\pd$ density as a translated rank-1 product of vectors (since it has one $k\times k$ elevated submatrix with uniform signal). Therefore we will simply take the \emph{Kronecker product} to result in a $m^2\times m^2$ matrix, which creates at the $(i,j)$th columns $R_i^TR_j$ where $R_i$ are the rows of the $m\times m$ matrix. 

However, the canonical rank-1 $\pd$ formulation is \emph{not centered}, having zeroes everywhere outside of the planted submatrix and elevated ones signal inside.
To map to this instance, we first need to transform the centered signal $R\to \frac{1}{\gamma}(R+\gamma)$ so that we get $\frac{m}{r}$ ones in an all-zero vector for each row in $R$ before taking the rank-1 product to get a $\frac{m}{r}\times \frac{m}{r}$ submatrix of ones inside $m\times m$ zeroes. This would make sure that the design matrix has exactly two different values. Unfortunately, doing so results in a product matrix that is guaranteed to have a large operator norm (since the output $\pd_D$ is easy), explained intuitively because now our matrix is not centered. 

To obtain a tighter spectral radius, it is natural for us to recenter the product matrix so that it has zero mean per column, corresponding to exactly $\pds$. This provides a justification from a design matrix perspective of why $\pds$ is probably harder than $\pd$: \emph{re-centering the design matrix decreases spectral norm, which results in a higher signal strength at the output.}
\begin{lemma}[Construction of (fixed size) random $K_{m}^{1/r}$]\label{RegularBig}
For given $\alpha$, exist absolute constant $C_{\alpha}>0$, such that for every $m>r>2$ where $r<m^{1-\alpha}$ divides $m$, there exist $m$ subsets $A_1, A_2,\dots, A_m$ of $[m]$ such that $|A_i|=\frac{m}{r}$, and that the $m^2\times m^2$ matrix $K_{(ij), (kl)}: i, j, k, l\in[m]$ defined as
$K_{(ij), (kl)}=\mu\sqrt{\frac{r}{m}}\cdot (1_{k\in A_i\, \text{and }l\in A_j}\cdot \frac{r}{m} -\frac{1}{mr})$ has largest singular value at most $1$.
Specifically,
$$K_{m}^{1/r}:=K=\mu\sqrt{\frac{r}{m}}\left[(R+\frac{1}{\sqrt{mr}}J)\otimes (R+\frac{1}{\sqrt{mr}}J)-\frac{1}{mr}J\otimes J\right]$$
where $J$ is the all-one matrix and $R$ satisfies the criteria from the previous lemma ($\mu=(C_{\alpha}+1)^{-2}\in \Theta_m(1)$).
With probability $1-o_m(1)$ we can find a satisfying assignment in polynomial time.
\end{lemma}


\section{Hardness of Detection in Mean-corrected Null}\label{thm1}
We are now ready to state hardness for the degree-1 corrected null hypothesis testing problem $\pds$ by constructing an average case mapping. We refer to \sref{Figure}{fig: pdsred} (\sref{Theorem}{PDSreduction}) for the full reduction.

\begin{theorem} [Lower bounds for efficient $\pds$ detection]\label{PDSbound}
Consider hypothesis testing $\pds$ for $H_0: G(n, p_0)$ versus $H_1: \pd(n, k, q, p)$ where $p_0=p-(\frac{n^2}{k^2}-1)\gamma=q+\gamma$. Let parameters $p_0\in(0, 1)$, $\alpha\in[0,2), \beta\in (0,1)$ and $\beta<\frac{1}{2}+\frac{2}{3}\alpha$. There exists a sequence $\{(N_n, K_n, p_n, q_n)\}$ of parameters such that:
\begin{itemize}
\item The parameters are in the regime $p-q\in \tilde\Theta(N^{-\alpha})$, $K\in\tilde\Theta(N^\beta)$. Formally,
$$\lim_{n\to\infty}\frac{\log p_n-q_n}{\log N_n}=-\alpha,\;\;\;\lim_{n\to\infty}\frac{K_n}{N_n}=\beta,\;\;\;\lim_{n\to\infty}\frac{\log(p_n-q_n)^{-1}}{\log N_n}=\alpha.$$
\item For any sequence of (randomized) polynomial-time tests $\phi_n : \mathcal{G}_{N_n} \to \{0, 1\}$, the asymptotic
Type I+II error of $\phi_n$ on  $\pds(N_n, K_n, p_n, q_n)$ is at least 1 assuming the $\PCR$ conjecture.
\end{itemize}

\end{theorem}

Furthermore, we note that there exists a matching upper bound for $\pds_D$ based on the empirical variance of degrees (see \sref{Proposition}{UpPDS} for the precise result). 

Recall that as discussed in \sref{Section}{Defs}, any recovery oracle (on $H_1=\pds_{H_1}$, which is the same as $\pd_{H_1}$) detects between $H_0: G(n, p_0)$ versus $H_1: \pd(n, k, q, p)$ on its supported parameters, implying a natural upper bound on the decision problem.
Combining \sref{Theorem}{PDSbound} with \sref{Lemma}{lem: oracle to detection}, we obtain our main (lower bound) result for the signal strength required for recovery.
\begin{corollary} [Recovery Hardness for PDS] \label{RECOVBOUND}
Let parameters $p_0\in(0, 1)$, $\alpha\in[0,2), \beta\in (0,1)$ and $\alpha< \beta<\frac{1}{2}+\frac{2}{3}\alpha, $. Then for any $p_0\in (0,1)$ there exists a sequence $\{(N_n, K_n, p_n, q_n)\}$ of parameters such that the following holds:
\begin{itemize}
\item The parameters are in the regime $\gamma:=|p-q|\in \tilde\Theta(N^{-\alpha})$, $K\in\tilde\Theta(N^\beta)$.
\item For any sequence of (randomized) polynomial-time algorithm  $\phi_n : \mathcal{G}_{N_n} \to {[N_n]\choose K_n}$, $\phi_n$ cannot achieve asymptotic exact recovery on $\pd(N_n, K_n, p_n, q_n)$ assuming $\PCR$.
\end{itemize}
\end{corollary}
We remark that the constraint $\alpha<\beta$ comes from the fact that recovery is statistical impossible at $\alpha\geq \beta$ (see \sref{Theorem}{PDSstatrec}). For completeness, we refer to the appendix (\sref{Theorem}{StatBound}) for an extended discussion on the statistical boundaries associated. Moreover, we remark that in light of our recovery to detection reduction framework, the detection-recovery gap can in fact be viewed as a detection ($\pd$) - detection ($\pds$) gap.

\section{Hardness of Refutation}
\subsection{Detection hardness for ISBM}
As before, given that a refutation blackbox only operates on $H_0$, we want to find some ``quiet" distribution $\tilde H_1$, such that it has the correct valuation but is hard to distinguish from a null instance. We will propose the ISBM model \eqref{ISBM} in this section as a qualifying planted distribution. Due to the rank-1 nature of its bias structure, it is easy to construct design matrices by just taking the per-column rank-1 product from \sref{Lemma}{RegularMatrix}, hence hardness result can be proven similar to 
\sref{Theorem}{PDSreduction} with a reduction. As in the proof of reduction to $\pds$, we can then generalize to the complete boundary in ISBM detection, leading to refutation hardness. This is an extension of Theorem 3.2 in \cite{BB20} where their (deterministic rotation kernel) reduction only works with a number-theoretic constraint restricting the parameters. Our results extend to the full boundary line by the regular concentration lemma on random matrices.
\begin{theorem}[Hardness of detection in ISBM]\label{ISBMBound}
Consider hypothesis testing $\textup{ISBM}_D$ \eqref{ISBM} where $k=n/r$ is the planted size. Let parameters $p_0\in(0, 1)$, $\alpha\in[0,2), \beta\in (0,1)$ and $\beta>\frac{1}{2}-\alpha$. There exist a sequence $\{(N_n, R_n, P_{11}^{(n)}, P_{12}^{(n)},P_{22}^{(n)})\}$ of parameters such that:
\begin{itemize}
\item The parameters are in the regime $|P_{11}-P_{22}|\in \tilde\Theta(N^{-\alpha})$, $R\in\tilde\Theta(N^\beta)$.
\item For any sequence of (randomized) polynomial-time tests $\phi_n : \mathcal{G}_{N_n} \to \{0, 1\}$, the asymptotic
Type I+II error of $\phi_n$ on the decision problems $\isbm_D(N_n, R_n, P_{11}^{(n)}, P_{12}^{(n)},P_{22}^{(n)})$ will be at least 1 assuming the $\PCR$. 
\end{itemize}
\end{theorem}
\subsection{Refutation hardness for planted dense subgraph in $G(n, p)$}
Equipped with the hardness results in $\isbm$, which has a large dense subgraph and thus can be used as a candidate $\tilde H_1$ in refutation, we obtain the formal refutation hardness results similar in how we showed recovery hardness from a reduction with refutation (recovery) oracle:

\begin{theorem}[Hardness in refutation of PDS in the dense regime]\label{REFUTEBound}
Consider the refutation problem for $H_0: G(n, p_0)$ and val function $v(G)$ defined as the edge density of the largest $k-$subgraph. Let parameters $p_0\in(0, 1)$, $\alpha\in[0,2), \beta\in (0,1)$ and $\beta>\frac{1}{2}-\alpha$. Then for any sequence of parameters $\{(N_n, K_n, p_1^{(n)}\}$ satisfying:
\begin{itemize}
\item The parameters are in the regime $p_1-p_0\in \tilde\Theta(N^{-\alpha})$, $K\in\tilde\Theta(N^\beta)$.
\item No sequence of (randomized) polynomial-time algorithms $\phi_n$ can achieve refutation with asymptotic successful probability strictly above $0$. 
\end{itemize}
\end{theorem}
Finally, we note that a matching (computational) upper bound can be constructed via a semi-definite programming relaxation (see appendix). Moreover, we also show that the statistical boundary for refutation lies exactly as that for recovery from applying a reduction to (statistical) recovery.
\begin{theorem}[Statistical bounds for refutation]\label{statrefute}
Consider refutation problem for $G \sim G(n, p_0)$ and val function $v(G)$ defined as the edge density of the largest $k-$subgraph. Assuming that $p_0$ is bounded away from 0 and 1, and $k\in\tilde\Theta(n^{\gamma})$ for some $\gamma\in(0.5, 1)$, then:
\begin{itemize}
    \item When $kD_{KL}(p\| p_0)\in\tilde \omega(1)$, the densest $k$ subgraph $\val(G)\leq p$ with probability $\to 1$.
    \item  When $kD_{KL}(p\| p_0)\in\tilde o(1)$, the densest $k$ subgraph $\val(G)\geq p$ with probability $\to 1$.
\end{itemize}  
\end{theorem}
\begin{remark}
The problem of densest-$k$-subgraph in $G(n, \frac{1}{2})$ was very recently solved in \cite{cheairi2022densest} with deep techniques from Bernoulli Disorder. However, here we can derive a log-optimal result using statistical reductions from recovery boundaries.
\end{remark}

\section{Biclustering and Biased Sparse PCA}\label{OtherGaps}
We point out a couple of other random models that have a detection hardness gap as an implication of PDS hardness guarantees. Those connections were first observed in \cite{CLR17,BBH18} but under the conjectural tight hardness bound and \cite{SchrammWein} with low-degree polynomials.

\paragraph{Bi-clustering} This model is planting a $k\times k$ (not necessarily principal) submatrix and can be formulated as the following Gaussian detection problem:
\begin{equation}\label{BC}
    H_0: Z\sim \mathcal{N}(0, 1)^{\otimes n\times n},\;\;\;\;\;H_1: Z\sim \mathcal{N}(0, 1)^{\otimes n\times n}+\lambda uv^T
\end{equation}
where $u, v\sim \Bern(k/n)^{\otimes n}$ (or uniform from all subsets of size $k$) independently. The recovery problem is to localize the latent vectors $u, v$ given an instance $Z\sim \mathcal{N}(0, 1)^{\otimes n\times n}+\lambda uv^T$, and the refutation is to refute submatrices with large mean.
\paragraph{Biased SPCA} Consider the \textit{spiked covariance model} where $v$ is a $k$-sparse unit vectors with non-zero entries equal to $\pm\frac{1}{\sqrt{k}}$:
\begin{align*} 
&H_0: X_1, X_2, \dots, X_n \sim \N(0, I_d)^{\otimes n} \quad \text{and} \\
&H_1 : X_1, X_2, \dots, X_n \sim \N\left(0, I_d + \theta vv^\top\right)^{\otimes n} \text{ where } \left|{\| v \|_0^+-\frac{k}{2}}\right|>\delta\cdot k. \numberthis{\label{BSPCA}}
\end{align*}
The recovery task is to estimate $\text{supp}(v)$ given observations $X_1, X_2, \dots, X_n$ sampled from $H_1$. Specifically for this variant where the sum test can be shown optimal for detection, our result implies a detection-recovery gap which is lacking in its general unbiased form.

\section{Open Problems}
We point out two open problems related to our work:
\begin{enumerate}
    \item Construct ``quiet" $H_0$ hypotheses without any dense subgraphs that are hard to distinguish from $\pd$ in order to resolve \sref{Conjecture}{conj: rec}. This would also imply a \emph{detection-certification gap} as well as \sref{Conjecture}{conj: rec} itself.
    
    \item Can one can construct the inverse of the reduction of \sref{Remark}{rmk:1}, from a binomial version of $\pd$ to the fixed sized $\pd$? This would show equivalence of the binomial and fixed versions.
\end{enumerate}

\acks{This work was supported in part by NSF CAREER award CCF-1940205.}

\bibliography{ref.bib}

\newpage
\appendix

\section{Notations and Preliminaries}\label{sec: prelim}

We briefly introduce the notations. We use $\mathcal{L}(X)$ to denote the law of a random variation $X$, $\dTV, D_{KL}, \chi^2$ to denote the total variation distance, KL-divergence, and $\chi^2$ divergence. Specifically we shorthand $d(\Bern(p), \Bern(q)):=d(p, q)$ for Bernoullis with bias $p, q$. We use the $\tilde O(\cdot)$ notation to denote big-O ignoring log-factors. For instance, $r\in \tilde \omega(n)$ means $r\in \omega(n\log^{k} n)$ for any constant $k$, $r\in \tilde \Omega(n)$ means $r\in \Omega(n\log^{k} n)$ for some $k$, and $\tilde O$, $\tilde o$ likewise. Specifically, $\tilde \Theta(n)=\tilde O(n)\bigcap\tilde \Omega(n)$. We use $\prod_i P_i$ to denote the tensor product of distributions, specifically $P^{\otimes k}=\prod_{i=1}^{k} P$.

\spac For a given partition $F$ of $[n]$ to $k$ sets, we use $\mathcal{U}_n(F)$ to denote the uniform distribution of $k$-subsets of $[n]$ with each element in one of $F_i$. We let $\textit{Unif}_n(k)$ to denote the uniform distribution over all $k$-subsets of $[n]$. For a planted structure distribution, we use $\mathcal{M}_{A\times B}(S\times T, P, Q)$ to denote planted structure on community $A\times B$ with a planted submatrix $S\times T$ where the in-community entries sampled from $P$ and otherwise from $Q$. Specifically, if $A=B$ and $S=T$ are unknown sampled from $\mathcal{P}$, denote $\mathcal{M}_{A\times B}(\mathcal{P}, P, Q):=\mathbb{E}_{S\sim \mathcal{P}}\left(\mathcal{M}_{A\times B}(S\times S, P, Q)\right)$ the symmetric planting.

\spac We use $A\otimes B\in \mathbb{R}^{n^2\times n^2}$ for matrices $(A, B)\in (\mathbb{R}^{n\times n}, \mathbb{R}^{n\times n})$ to denote the Kronecker product between $A, B$. We usually parameterize indices of $A\otimes B$ by a pair $(ij): i, j\in [n]$ such that $(A\otimes B)_{(ij), (kl)}=A_{ik}B_{jl}$. Fixing $i, j$ and laying out the row of $A\otimes B$ as a $n\times n$ matrix, it is exactly $A_{i, \cdot}^{T}B_{j, \cdot}$ the product of two row-vectors. 

\spac We then introduce the following (common) lemmas as preliminaries. Let $f$ be a Markov transition kernel and $P$ be any distribution we denote the law of $f(P)$ the push-forward. We also use sets in $V\in 2^{[n]}$ and vectors $v\in \{0,1\}^{n}$ interchangeably, and $\pd(n, S, p, q)$ to be the planted dense subgraph instance conditioned on planted location at set $S$.
\begin{lemma}[Data Processing Inequality] Let $\mathcal{f}$ be a Markov transition kernel and $A, B$ be two distributions, then:
$$\dTV(f(A), f(B))\leq \dTV(A, B).$$
\end{lemma}
\begin{lemma}[Tensorization of TV] \label{Tensor}
Let $P_i, Q_i$ be distributions for $i=1,2,\dots, n$. Then:
$$\dTV(\prod P_i, \prod Q_i)\leq \sum_i \dTV(P_i, Q_i).$$
\end{lemma}
\begin{lemma}[Accumulation of TV distance]\label{TVACC}
Consider a finite set of sequential functions on distributions $\mathcal{A}_i: i=1,2,\dots, k$. Assuming one has distributions $P_0,P_1, P_2,\dots, P_k$ such that:
$$\dTV(\mathcal{A}_i(P_{i-1}), P_i)\leq \epsilon_i$$ for all $i=1,2,\dots, k$, then we have:
$$\dTV(\mathcal{A}_{k}(\dots\mathcal{A}_1(\mathcal{A}_1(P_0))\dots), P_k)\leq \sum_i\epsilon_i.$$
\end{lemma}
The last lemma comes directly from data processing and induction. Next, we present a couple of lemmas on Bernoulli distributions.

\begin{lemma}[KL divergence between Bernoullis]\label{KLAPPROX}
Assume a sequence of $\{p_n\}$ and $\{q_n\}$ such that $q_n<p_n<cq_n$, $1-q_n<c(1-p_n)$ for some constant $c$, then:
$$D_{KL}(p\|q):=D_{KL}(\Bern(p)\|\Bern(q))=\Theta(\frac{(p-q)^2}{q(1-q)}).$$
\end{lemma}
\begin{proof}
Note that the quantity $\frac{(p-q)^2}{q(1-q)}$ is the $\chi^2$ divergence between two Bernoulli, which dominates the KL divergence. For the other side, note that on the support of these two distributions ($\{0,1\}$) their ratio of density is bounded. Thus by a reverse Pinsker's inequality the result follows. 
\end{proof}
\begin{lemma}[TV divergence between Binomials]\label{TVBIN}
Consider two parameters $p, q\in(0,1)$, then:
$$\dTV(\Bern(q)^{\otimes n}, \Bern(q)^{\otimes n})\leq \sqrt{\frac{n(p-q)^2}{2q(1-q)}}.$$
\end{lemma}
\begin{proof}
This comes directly from the Pinsker's inequality on TV, KL, and $\chi^2$ divergences:
\begin{align*}
    \dTV(\Bern(q)^{\otimes n}, \Bern(q)^{\otimes n})&\leq \sqrt{\frac{D_{KL}(\Bern(q)^{\otimes n},\Bern(q)^{\otimes n})}{2}} \\
    &=\sqrt{\frac{nD_{KL}(\Bern(p), \Bern(q))}{2}}
    \\&\leq \sqrt{\frac{n\chi^2(\Bern(p), \Bern(q))}{2}}=\sqrt{\frac{n(p-q)^2}{2q(1-q)}}
\end{align*} due to $2\dTV^2\leq D_{KL}\leq \chi^2$ and the factorization of $D_{KL}$ for independent distributions.\end{proof}
\section{Reductions to Detection}\label{reducs}
In this section we point out that all of the inference variants considered, detection is (almost) the weakest version of all. This can be viewed from the perspective of reductions where a blackbox for a different task implies a blackbox for detection. Such reductions were discussed in 
\cite{HWX15,BHK16,BBK,BBH18}. Here we re-formulate the necessary proofs:

\begin{lemma}[Refutation implies detection]\label{refutetodetect}
Consider two hypotheses $H_0, H_1$ and valuation function $\val$ with separation thresholds $\epsilon$ and gap $\delta$. If there is an efficient refutation blackbox $A$ with asymptotic success probability $p=\lim_{n\to\infty}p^{(n)}>0$, then (weak) detection is computationally possible.
\end{lemma}
\begin{proof}
Consider the canonical form of refutation as described in \sref{Section}{Defs} with a polynomial-timed refutation blackbox $A$ having success probability $p>0$. We show that $A$:\begin{itemize}
    \item When $G\sim H_0$, $A(G)=0$ with probability at least $p\cdot P(\val(G)<\epsilon-\delta|G\sim H_0)$, thus the Type I error is at most $1-p\cdot P(\val(G)<\epsilon-\delta|H_0)$.
    \item When $\val(G)>\epsilon+\delta$ the output is always 1, thus the Type II error is at most the probability of a low valuation $P(\val(G)>\epsilon+\delta|H_1)$.
\end{itemize}
Therefore, the sum of errors is bounded above by:
$$1-p\cdot P(\val(G)<\epsilon-\delta|H_0)+P(\val(G)>\epsilon+\delta|H_1)$$
Note that as $n\to\infty$,
$$P(\val(G)<\epsilon-\delta| H_0)\to 1, \;\;\;\;P(\val(G)>\epsilon+\delta|H_1)\to 0.$$
Therefore, the detection error of this blackbox is bounded above by $1-p^{(n)}<1$ in the limit. This implies that it returns a better-than-random detection asymptotically.
\end{proof}
\begin{lemma}[Recovery implies detection]\label{recovtodetect}
Fix two hypothesis $H_0, H_1$ with valuation function $\val_s$ that can be computed in polynomial time given key $s$ and define $\val(G)=\max_s\val_{s}(G)$. 

Assuming $\val(G\sim H_0)\leq \epsilon-\delta$ with high probability and suppose that a polynomial time oracle (on input $G$) generates a key $k$ such that $\val_k(G)> \epsilon$ with high probability over $G\sim H_1$, then detection is possible in polynomial time equipped with such key oracle.
\end{lemma}
\begin{proof}
Consider the alternate valuation function $\val'=\val_k$, which can be computed in polynomial time by first asking $k(G)$ from the given oracle. From the previous lemma and the separation conditions we know that:
\begin{enumerate}
    \item For $G\sim H_0$, $\val'(G)\leq \val(G)\leq \epsilon-\delta$ with high probability.
    \item For $G\sim H_1$, $\val'(G)> \epsilon$ with high probability.
\end{enumerate}
Therefore, $\val'$ is a polynomial time valuation, which obviously implies that refutation on this blackbox can be done in polynomial time. By \sref{Lemma}{refutetodetect}, we have the desired conclusion. 
\end{proof}
\section{Different Varieties of Recovery}
\label{recovs}
Consider the notion of \emph{minimal recovery} of strength $\alpha>0$, which is outputting a guess $\hat P$ for the planted location such that $$\lim_{n\to\infty}\frac{(\log k)^{\alpha} \mathbb{E}[|\hat P\bigcap P|]}{k}\geq 1.$$ Specifically, weak recovery is just partial recovery of strength $1$ and partial recovery implies minimal recovery of strength $\alpha\to 0$. We can go on to prove that with partial recovery one can achieve precise recovery with only sub-polynomial signal boost. This means that the PDS recovery conjecture can be weakened to only assume hardness for \emph{minimal} recovery. The following lemma applies both statistically and computationally; we will use it as a crucial reduction step to \sref{Theorem}{statrefute}.
\begin{lemma}[Minimal recovery implies exact recovery]\label{partialtofull}
For the $\pd(n, k, p, q)$ recovery problem when $p, q$ are bounded by away by zero and one. If one can achieve minimal recovery with strength $\alpha>1$ on a sequence of parameters $(N_n, K_n, P_n, Q_n)$ in polynomial time where $K_n\in \tilde{\omega}(\sqrt{N_n})\bigcap o(N_n^{\gamma})$ for some exponent $\gamma\in (\frac{1}{2}, 1)$, then one can achieve exact recovery on a modified sequence of parameters $(N_n, K_n, P_n, Q'_n)$ where $Q'_n$ satisfies $$D_{KL}(P_n\|Q_n)=\Theta((\log k)^{2\alpha}D_{KL}(P_n\|Q'_n)).$$
\end{lemma}
\begin{proof}
The critical component here lies in a subroutine \textsc{Graph-Clone} (Lemma 5.2 in \cite{BBH19}) in which we generate independent graph instances conditioned on planted instance locations. The lemma can be read off as the following form:\begin{itemize}
    \item Suppose we have a hidden planted location $\eta$, and a one-time sampler from the planted distribution $\mathcal{M}_{[n]\times [n]}(\eta\times \eta, \Bern(p), \Bern(q))$, then we have a one-time sampler from the tensor product $\mathcal{M}_{[n]\times [n]}^{\otimes 2}(\eta\times \eta, \Bern(p), \Bern(Q))$ where $Q=1-\sqrt{(1-p)(1-q)}$. Specifically, the divergence measure $\chi^2(p, Q)>\frac{1}{2}\chi^2(p, q)$.
    \item[Corollary:] Suppose we have a hidden planted location $\eta$ and as above a one-time sampler, then we can have a sample generated from $\mathcal{M}_{[n]\times [n]}^{\otimes 2}(\eta\times \eta, \Bern(p), \Bern(Q))$ where $\chi^2(p, Q)>\frac{1}{2t}\chi^2(p, q)$.
\end{itemize}
Note that in the case when $p, q$ are bounded away from zero and one, $D_{KL}(p\|q)\sim \chi^2(p, q)\sim (p-q)^2$ are of the same order. Therefore, with the cost of reducing $(\log k)^{2\alpha}$ in the distance we can generate one instance from $\mathcal{M}^{\otimes (\log k)^{2\alpha}}_{[N]\times [N]}(\eta\times \eta, \Bern(P), \Bern(Q))$ from a single instance of the original $\mathcal{M}_{[N]\times [N]}(\eta\times \eta, \Bern(P), \Bern(Q'))$.

\spac Our assumption also states that we have a black box to perform minimal recovery of strength $\alpha$ on $\mathcal{M}_{[N]\times [N]}(\eta\times \eta, \Bern(P), \Bern(Q))$, and we arrive at $\log^{2\alpha} k$ estimates of the planted $\eta$, denoted by $\hat\eta_1,\dots, \hat\eta_{\log^{2\alpha}k}:=\hat\eta_{r}$. Moreover, given that the cloned copies are \emph{independently generated} conditioned on $\eta$ and hence so are those $\hat\eta_i$'s, we wish to reconstruct $\eta$ through those independent estimate $\hat\eta_i$'s.

\spac Note that it is safe to assume that any black-box takes in the input unlabeled, because we can apply a hidden permutation $\pi$ to the graph and feed it to the black-box instead, we know that for each index $i\in \eta$, the probabilty that it lies in a $\hat \eta$ is exactly $E(\hat \eta\cdot \eta)/k\sim \log^{\alpha} k$ where expectation ranges over $\mathcal{M}_{[N]\times [N]}(\eta\times \eta, \Bern(P), \Bern(Q))$. For each node not in $\eta$, the probability of $\hat \eta$ hitting it is at most $\frac{k}{n-k}\leq \frac{2k}{n}$. Therefore, if we compile a histogram of $\hat \eta_i$ hits, we have $k$ copies of $\Binom(r, \sqrt{r^{-1}})$ and $n-k$ copies of (at most) $\Binom(r, \frac{2k}{n})$.

\spac We now only need to show that when $r\in\Theta((\log k)^{2\alpha})$, the two distributions (smallest from $\eta$ and largest from $\overline \eta$) separates with probability $\to 1$. Note that the probability that $\Binom(r, 2k/n)$ is at least constant $C$ is:
$$\P(\Binom(r, 2k/n)>C)\leq r\cdot r^C\cdot (\frac{2k}{n})^C\lesssim (\log n)^{2(C+1)\alpha}n^{-C(1-\gamma)}$$
Pick any $C(1-\gamma)>1$, then the above probability goes to $\tilde o(n^{-1})$, and a union bound over all vertices in $\overline{\eta}$ says that with probability $1-o_n(1)$ all counts in that group is bounded by constant $C$.

\spac Now we consider the group of nodes in $\eta$. The probability of one being bouneded by $C$ is 
$$\P(\Binom(r, \sqrt{r^{-1}})<C)\leq C{r\choose C}(1-\sqrt{r^{-1}})^{r-C}\leq Cr^C\exp\left(-(r-C)\sqrt{r^{-1}}\right)$$because $1-x<e^{-x}$, and the union bound says that the minimum for counts in $\eta$ is at most:
$$k\P(\Binom(r, \sqrt{r^{-1}})<C)\lesssim kr^{C}\exp\left(-\frac{1}{2}\sqrt{r}\right)=\exp\left(\log k-\frac{(\log k)^{\alpha}}{2}+O(\log r)\right) $$
assuming that $\alpha>1$, the above goes to $0$ as $k\to\infty$.

\spac Therefore, if we sum over statistics for $\hat \eta_i$'s we get that the entries in $\eta$ goes above any constant with probability $\to 1$ whereas with high probability the other entries are bounded by a constant, and hence precise recovery is achievable via the most popular nodes. Moreover, this entire procedure applies in polynomial time.
\end{proof}
We further comment that the case of when $p, q$ are not bounded away by one (``dense") are similar, but require some further bounds on the exponents. Here we only present proof of this dense regime. Moreover, the condition that $k<n^{\gamma}$ instead of $k\in o(n)$ is also only needed for minimal recovery (not required for a reduction from partial to exact), but for our purposes this is a fine assumption to make. The condition $\alpha>1$ is to some extent unnecessary either because recovery of $\alpha$ implies recovery of $\alpha^{+}$ for any $\alpha^{+}>\alpha$ ($\alpha>1$ is only needed for convenience in expressing signal decay).

\spac Moreover, note that the non-homogenity of planted instance in \sref{Theorem}{PARTRecov} means that the direct reduction in \sref{Lemma}{partialtofull} does not apply. In fact, it will be easy to see that partial recovery for our instance in \sref{Theorem}{PARTRecov} can be implied by PDS recovery.

\spac For completeness of arguments (which will be useful in statistical bound for refutation), we also present a statistical condition for (exact) PDS recovery in below.
\begin{theorem}[PDS recovery -- Theorem 2 in \cite{HWX16b}]\label{PDSstatrec}Again consider the settings (and parameters correspondence) as above, then exact recovery is statistically possible if:
$$\frac{k D_{KL}(p\| q)}{\log n}> C$$and impossible if
$$\frac{k D_{KL}(p\| q)}{\log n}< c$$for some absolute constants $c, C$.
\end{theorem}

\subsection{Evidence of recovery hardness}
 We further remark that with some \emph{relaxed} condition on the signal, one can prove tight recovery hardness for some models on $\DkS$ (Densest $k$-Subgraph) valuation that resembles PDS in structure.
\begin{theorem}\label{PARTRecov}
For any $k_n\in \omega(\sqrt{n}), p_n\in (0, 1)$, there exists a symmetric edge density matrix on subgraphs $D_n\in \mathbb{R}^{k_n\times k_n}$, such that the row (and column) sums of $D$ are uniformly $k_n\cdot\lambda_n$ and the graph constructed by:
\begin{enumerate}
    \item On $G\sim G(n, p_n)$, randomly select a subset $S$ of vertices of size $k_n$. Choose a random bijection $\pi$ from $S$ to $[k_n].$
    \item For the nodes $u, v \in S$, resample $uv$ with probability $D_{\pi(u)\pi(v)}+p$.
\end{enumerate}
And if $\lim\sup_{n\to\infty}\frac{k_n^2}{n}\frac{\lambda_n^2}{p_n(1-p_n)}\in \tilde{o}(1)$, no (randomized) polynomial algorithm can achieve exact recovery on the planted instance, even given the knowledge of $D$, assuming the planted clique conjecture with $p=1/2$.
\end{theorem}
It is not hard to check that for this general model (where recovery is at least as hard as PDS), one can still find a log-optimal algorithmic matching upper bound. Specifically, consider simply taking the $k$ most popular nodes (those with the highest degrees), then as long as the ratio $\frac{k^2\lambda^2}{np(1-p)\log n}\to\infty$, the output satisfies \textit{exact recovery} criteria. This result also suggest that, if PDS recovery is indeed easier than conjectured, the algorithm must use heavily the community structure.

\begin{proof}
We start from the fact that, \textit{recovery for \textup{PC} is hard at the regime when detection is hard}, which is a direct implication of the PC conjecture and \sref{Lemma}{recovtodetect}. Consider the following procedure applied on a graph $G\sim \pd(n, k, 1, 1/2)$ to obtain $G'$:\begin{enumerate}
    \item Add $(t-1)k$ vertices to $G$ that will be part of the (new) planted structure where $t>1$ is a specified parameter such that the total planted size is $tk$. 
    \item For the $(t-1)k$ extra vertices, connect each pair with probability $\frac{t}{2(t-1)}$. Connect each edge between the original $n$ vertices to the new $(t-1)k$ vertices with probability $1/2$.
    \item Permute the nodes in $G'$ randomly.
\end{enumerate}
Under this reduction, consider the new planted density matrix in $\mathbb{R}^{tk\times tk}$ where a $\mathbb{R}^{(t-1)k\times (t-1)k}$ principal submatrix is $\frac{1}{2(t-1)} J_{(t-1)k}$ and the other $k\times k$ principal submatrix has all entries $1/2$. The recovery hardness comes from the fact that even if the blackbox knows the exact location where our planted $(t-1)k$ nodes are, it can still not precisely recover the original $k$ vertices in the planted clique instance with high probability, thus strong recovery is impossible. 

\spac Note that this satisfies the row-column sum constraint where the expected degree of each planted node is exactly $\frac{n+(t+1)k}{2}$ and the planted structure lifted each node's degree by $k/2$. Consider the distribution of the degrees for a node not in planted structure (which is $d\sim \Binom(n+(t-1)k, 1/2)$) and the node inside planted structure which is either $d\sim k+\Binom(n+(t-2)k, 1/2)$ or $d\sim \Binom(n, 1/2)+\Binom((t-1)k, 2t/(t-1))$ depending on which part of the planted set. The separation of the first distribution (null) with the later two (latent) follows immediately by a very simple Chernoff Bound when $k\in\tilde \omega(\sqrt{n})$.
\end{proof}

\section{Proofs for Design Matrices}
\subsection{Proof of \sref{Lemma}{RegularMatrix}}
\begin{proof}
Note that the adjacency matrix of the sampled directed graph $A$ is not symmetric. However, we do know that the operator norm equals to the largest singular value of $$(A-\frac{d}{n}\one\one^T)(A^T-\frac{d}{n}\one\one^T)=AA^T-\frac{d^2}{n}\one\one^T$$where $d=n/r$ and $\one\in\mathbb{R}^{n\times 1}$ is the all-one vector.

\spac
We know that the largest eigenvalue of $AA^T$ is $d^2/n$ corresponding to the all one vector because it is a scaled doubly stochastic matrix. Therefore, from the Courant-Fischer Theorem we can show that the second largest singular value of $A$ which is the second largest eigenvalue of $AA^T$ is the largest eigenvalue of $AA^T-\frac{d^2}{n}\one\one^T$, which is the largest singular value of $A-\frac{d}{n}\one\one^T$.

\spac
Note that Theorem.B of  \cite{RegGraphSpectral} asserts that under the conditions in the lemma, the said quantity is bounded by $C\sqrt{d}$ with probability $1-o_n(1)$. Therefore, our constructed $R$, which is exactly $\sqrt{\frac{r}{n}}(A-\mathbb{A})=\sqrt{d^{-1}}(A-\frac{d}{n}\one\one^T)$ has max singular value at most $C$ with probability $1-o_n(1)$.
\end{proof}
\subsection{Proof of \sref{Lemma}{RegularBig}}
Note that we can sample from directed regular graphs efficiently by the ergodicity of a simple edge-flipping Markov process (\cite{MixReg1,MixReg2}), and hence we have the following.
\begin{proof}
Take $R$ from the previous lemma.
Note that the top singular value of a matrix is in fact sub-additive, and the Kronecker product is a linear operator that preserves the product of the operator norm of a matrix. We have:
    $$\sigma(\mu^{-1}\sqrt{\frac{n}{r}}K)\leq \sigma(R\otimes R)+\frac{2}{\sqrt{nr}}\sigma(R\otimes J)\leq C^2+\frac{2C}{\sqrt{nr}}\sigma(J)=C^2+2C\sqrt{\frac{{n}}{r}}$$because the top eigenvalue of $J=\one\one^T$ is exactly $n$.
    Therefore $\sigma(K)\leq 1$ for any $R$ satisfying the criteria of \sref{Lemma}{RegularMatrix}.
    
    \spac Finally, note Theorem 1 in \cite{MixReg1} states that the switch Markov Chain, on which the unique stationary distribution is the uniform distribution over all directed $d-$regular graphs, is fast-mixing, and \sref{Lemma}{RegularMatrix} still holds if the sampling condition is approximate in $L1$. Therefore, in polynomial time we can find one candidate $K$ satisfying our lemma above. Moreover, note that if we sample $O(n)$ times independently, the probability of failure becomes exponentially small.
\end{proof}
\section{Proofs for Recovery}\label{nots}
As a starter to detection problems in $\pds$, we present a simple $\pds_D$ upper bound by a degree 2 polynomial test extending Proposition B.4 in \cite{SchrammWein}.
\begin{proposition}[Upper bound on $\pds_D$] \label{UpPDS}
Consider the degree corrected $\pds$ model with planted subgraph size $k$ and average edge probability $0<q<p<\frac{2}{3}$, then as long as the product ratio $\frac{k^3}{n^{1.5}}\cdot\frac{(p-q)^2}{q(1-q)}\in\omega_n(1)$, one can computationally efficiently resolve hypothesis testing for $\pds$.
\end{proposition} 
The proof goes by considering the statistics $f=\sum d_i^2$ where $d_i$ are the degrees of $G$ and computing the mean different over variance. The upper bound gives a boundary strictly between the sum-test level for PDS and the spectral recovery level (Kesten-Stigum threshold), suggesting some consideration into the ``community" structures compare to a vanilla sum test.
\begin{proof}
Consider the test statistics $f(G)=\sum d_i^2$ where $d_i$ are the (independent) degrees. We show that there exist $\tau$ such that $\P_{H_0}(f(G)>\tau)+\P_{H_1}(f(G)<\tau)\to 0$ as $n\to\infty$.

\spac Firstly, consider what happens to the degrees under $H_0$: they are $n$ independent samples from $\Binom(n, p_0)$ with expectation given by
$$\mathbb{E}(f)=n\cdot \mathbb{E}_{x\sim \Binom(n, p_0)}x^2=n^2p_0(1-p_0)+n^3p_0^2=n^2p_0+(n^3-n^2)p_0^2$$
Similarly in $H_1$, there are $n-k$ nodes that are not in the planted set and their corresponding second moment of degree is:
$$\mathbb{E}(\sum_{i\not\in v}d_i^2)=(n-k)\mathbb{E}_{x\sim \Binom(n, q_1)}x^2=n(n-k)q_1+
\left((n-k)n^2-n(n-k)\right)q_1^2$$
and the $k$ nodes planted has:
\begin{align*}
    & \mathbb{E}(\sum_{i\in v}d_i^2)=k\mathbb{E}_{x\sim \Binom(n-k,\, q_1), y\sim \Binom(k,\, p_1)}(x+y)^2\\
    &=k\left(\left((n-k)q_1(1-q_1)+(n-k)^2q_1^2\right)+kp_1(1-p_1)+k^2p_1^2+2q_1p_1k(n-k)\right)\\
    &=k(n-k)q_1+k^2p_1+k((n-k)q_1+kp_1)^2-k(n-k)q_1^2-k^2p_1^2
\end{align*}
The difference between expectations in $H_0$ and $H_1$ is thus:
$$k((n-k)q_1+kp_1)^2+(n-k)(n^2-n-k)q_1^2-k^2p_1^2-n^2(n-1)p_0^2\in \Theta(k^3(p_1-q_1)^2).$$
Now we turn to estimating the variance of $f$. First consider the variance of $f\sim H_0$, which can be computed via the moments of binomial distribution\begin{align*}
    \var(f)&=n\cdot \var(d_i^2)\\
    &=\mathbb{E}_{x\sim \Binom(n, p_0)}(x^4)-\left(\mathbb{E}_{x\sim \Binom(n, p_0)}(x^2)\right)^2 \\
    &=n^2p_0(1-p_0)(1+(2n-6)p_0(1-p_0))\in\Theta(n^3p_0^2(1-p_0)^2)
\end{align*}
due to a simple computation $\mathbb{E}_{x\sim \Binom(n, p_0)}((x-np)^4)=np(1-p)(1+(3n-6)p(1-p))$.

\spac
For the variance of $f\sim H_1$, consider
\begin{align*}
    \var(f)&=(n-k)\cdot \var_{i\not\in v}(d_i^2)+k\cdot\var_{j\in v} (d_j^2)\\
    &\leq \Theta(n^3\left(q_1(1-q_1)\right)^2)+ k\cdot\mathbb{E}_{j\in v} \left((d_j-\bar d)^4\right)\\
    &\leq \Theta(n^3\left(q_1(1-q_1)\right)^2)+O(kn^2\left(q_1(1-q_1)\right)^2)\\
    &= \Theta (n^3p_1^2(1-p_1)^2)
\end{align*}
because of local inequality $(x+y)^4\leq 16(x^4+y^4)$.
Therefore, we know that $$\frac{\mathbb{E}_{H_1}(f)-\mathbb{E}_{H_0}(f)}{\sqrt{\var_{H_0}(f)+\var_{H_1}(f)}}\in O(\frac{k^3(p_1-q_1)^2}{n^{1.5}p_0(1-p_0)})=O((\frac{k^2}{n})^{1.5}D_{KL}(p_1\|q_1))$$and when this value is in $\omega(1)$, the two hypothesis can be separated (by, for instance, thresholding at $(\mathbb{E}_{H_1}(f)+\mathbb{E}_{H_0}(f))/2$).
\end{proof}

\subsection{Proof of reduction to $\pds$}
\begin{theorem} [Reduction to $\pds$]\label{PDSreduction}
Given any fixed constant $\alpha>0$. Let $N, k_0$ be parameters of planted clique graph size, $(n, k)$ be the target graph sizes where $\frac{n}{k}=:r<\left(\frac{N}{k_0}\right)^{1-\alpha}$.  We present the following reduction $\phi$ with absolute constant $C>1$:
\begin{itemize}
    \item \textup{Initial $k$-$\pd$ Parameters:} vertex count $N$, subgraph size $k_0\in o_N(N)$ dividing $N$, edge probabilities $0 < q < p \leq 1$ with $\min\{q, 1 - q, p-q\} = \Omega(1)$, and a partition $E$ of $[N]$. We further assume that $k_0\in o(\sqrt{N})$ holds (otherwise detection for the PDS problem will be easy).

    \item \textup{Target $\pds$ parameters:} $(n, r, k)$ where $r\in o(\sqrt{n})$ is a specified parameter, $k=n/r$ is the target subgraph size, and $n$ is the smallest multiple of $k_0r$ greater than $(1+\frac{p}{Q})N$ where $$Q=1-\sqrt{(1-p)(1-q)}+1_{p=1}(\sqrt{q}-1)$$is the cloned signal strength from pre-processing.
    
    \item \textup{Target $\pds$ edge strength:}
    $$\gamma=\mu(\frac{k_0r}{n})^{1.5},\;\;\;\;P_1=\Phi(\frac{(r^2-1)\gamma}{r^2}),\;\;\;\;\; P_2=\Phi(-\frac{\gamma}{r^2}),$$
    where $\mu\in(0, 1)$ satisfies that$$\mu\leq\frac{1}{12C\sqrt{\log(N)+\log(p-Q)^{-1}}}\cdot\min\{\log(\frac{p}{Q}), \log(\frac{1-Q}{1-p})\}.$$
    where $\gamma$ denotes the signal strength $\gamma=\Theta(D_{KL}(P_1\|P_2))$ roughly the KL-divergence between two output Bernoullis.
    \item Applying $\phi$ on the given input graph instance $G$  yields the following:
    \begin{align*}
        \dTV(\phi(G(N,\frac{1}{2})), G(n, \frac{1}{2}))&=o_n(1)\\
        \dTV(\phi(PC_{E}(N, k_0, \frac{1}{2})), \pd(n, k, P_1, P_2))&=o_n(1)
    \end{align*}
\end{itemize}
\end{theorem}
\begin{figure}
    \input{bigchunk}
    \vspace{-5mm}
    \caption{\footnotesize Reduction from $k$-PDS to $\pds$.}
    \label{fig: pdsred}
\end{figure}
\textit{Proof sketch:} Step by step, our proof proceeds from establishing the following lemmas for each step of our reduction in \sref{Figure}{fig: pdsred}, {a formal proof for the lemmas will be presented later}.
\begin{lemma} [$\textsc{To-}k\textsc{-Partite-Submatrix}$ $-$ Lemma 7.5 in \cite{BB20}]\label{TOBIPARTITE}
With the given assumptions, step 1 (denote as $\mathcal{A}_1$) of the reduction runs in poly$(N)$ time and it follows that:
\begin{align*}
    \dTV(\mathcal{A}_1(G(N, q)), \Bern(Q)^{\otimes n\times n})&\leq 4k_0\exp{(\frac{-Q^2N^2}{48pkn})}\\
    \dTV(\mathcal{A}_1(G(N, \mathcal{U}_N(E),p, q)), \mathcal{M}_{[n]\times [n]}(\mathcal{U}_n(S), p, Q)&\leq 4k_0\exp{(\frac{-Q^2N^2}{48pkn})}+\sqrt{\frac{C_Q k_0^2}{2n}}
\end{align*}
where $E$ is the partition of $[N]$ and $S$ is the partition of $[n].$
\end{lemma}
\begin{lemma} [Bernoulli Rotations for $\pds$]\label{BernRotation} Let $\mathcal{A}_2$ denote the output matrix $M$ from the second step of our reduction (before permutation). Suppose $S$ is a partition of $[n]$ to $k_0$ equal parts and planted set $|T\cap S_i|=1$ for all $i$. Let $M_i: S_i\to [n/k_0]$ be any fixed bijection. Let $K_{n/k_0}^{1/r}$ be the design matrix obtained from \sref{Lemma}{RegularBig} with embedded sets $A_1, A_2,\dots, A_{n/k_0}\subset [n/k_0]$, then the following holds:
\begin{align*}\dTV(\mathcal{A}_2(\Bern(Q)^{\otimes n\times n}), \;\mathcal{N}(0, 1)^{\otimes n\times n} )&=O(n^{-1})\\
\dTV(\mathcal{A}_2(\mathcal{M}_{[n]\times [n]}(\mathcal{U}_n(S), p, Q)),\; \mathcal{L}( \gamma\cdot X+\mathcal{N}(-\frac{\gamma}{r^2}, 1)^{\otimes n\times n}))&=O(n^{-1})
\end{align*}
where $X\in \mathbb{R}^{n\times n}$ is the random variable defined in each block $S_i\times S_j$ as a function of $T$:$$X_{S_i, S_j}=\left(
\text{$\mathds{1}(M_i^{-1}(A_{f(i)})\times M_j^{-1}(A_{f(j)}))$ where $f(i)=M_i(T\cap S_i)$}\right)$$
\end{lemma}
\begin{lemma} [Thresholding from Gaussians]\label{Gaussianization}
Let $\mathcal{A}_3$ be the final step from the above reduction, with the same notations as the previous lemma, then:
$$\mathcal{A}_3(\mathcal{N}(0, 1)^{\otimes n\times n})\sim G(n, 1/2)$$
$$\mathcal{A}_3(\mathcal{L}( \gamma\cdot X+\mathcal{N}(-\frac{\gamma}{r^2}, 1)^{\otimes n\times n}))\sim \pd(n, k, P_1, P_2).$$
\end{lemma}

As an important pre-processing step, we single out the steps in \sref{Lemma}{TOBIPARTITE} first, the proof in its exact form is deferred to \cite{BB20}. The general idea is that, to construct a bi-partite variant, after applying \textsc{Graph-Clone} to the instance and occupying the lower half of the adjacency matrix, we still need to figure out what happens in the diagonal. However, when we plant around $\sqrt{k}$ entries in a diagonal it's \emph{almost} the same as not planting anything in total variation, which means that we only need to blow up the size by a little bit to \emph{hide} the diagonal.

\spac 
After \sref{Lemma}{TOBIPARTITE}, we arrive at a bi-partite $k$-PDS instance with slightly different parameters, and we prove the following lemmas to complete the reduction.
\begin{proof}[\sref{Lemma}{BernRotation}]
We take a close look at what Bernoulli rotation produces for $H_1$. In the flattening step, we first define $k_0$ bijections $\pi_i$ from $S_i\to [n/k_0]$ (the order to be flattened). Looking at each submatrix block with a planted bit at $(T\bigcap S_i, T\bigcap S_j)$ is equivalently an instance of $$F_{S_i\times S_j}\sim \textup{PB}((n/k_0)^2, t, p, Q)$$ where the location indices are defined with $(i, j): i, j\in [n/k_0]$ and planted bit $$t=(\pi_i(T\bigcap S_i), \pi_j(T\bigcap S_j)):= (t_i, t_j).$$ Therefore, the output row of $K_{n/k_0}^{1/r}$ is precisely (indexed by $r, s$):$$K_{(t_i, t_j), (rs)}=\mu\left(\one\{r\in A_{t_i} \textup{ and } s\in A_{t_j}\}\cdot \sqrt{\frac{r^3k_0^3}{n^3}}-\sqrt{\frac{k_0^3}{rn^3}}\right).$$ After sending $\mathcal{A}(\cdot)_{(rs)}\to M_{\pi_i^{-1}(r), \pi_j^{-1}(s)}$, we know that 
$M\in\mathbb{R}^{(n/k_0)\times (n/k_0)}$ is a bi-partite matrix with $A_{t_i}\times A_{t_j}$ submatrix being elevated and (approximately) distributed as 
$$\mathcal{M}_{S_i\times S_j}(\pi_i^{-1}(A_{t_i})\times \pi_j^{-1}(A_{t_j}), \N((r^2-1)\gamma/r^2, 1), \N(-\gamma/r^2, 1)).$$with total variation loss at most $O((\frac{n}{k_0})^2R_{rk}^{-3})$ by  \sref{Lemma}{lem:bern-rotations}.

\spac For the other hypothesis $H_0$, simply note that the matrix gets sent to $\N(0, 1)$ independently for each entry and gets send to independent standard normal Gaussians. Therefore the rotation matches.

\spac Finally, note that in each block we differs from the target by at most $O(n^2R_{rk}^{-3})$ in $\dTV$, which results in at most $O(n^4R_{rk}^{-3})$ difference in $\dTV$ by the tensorization property. However, note that we can choose $R_{rk}$ to be any polynomial of $n$, and hence the Lemma holds.
\end{proof}
\begin{proof}[\sref{Lemma}{Gaussianization}]
Note that if we threshold at zero, then:
\begin{enumerate}
    \item $\N(0, 1)\to \Bern(1/2)$. 
    \item $\N(\mu, 1)\to \Bern(\Phi(\mu))$ for any $\mu$.
\end{enumerate}
Therefore we know that $\N(\frac{-\gamma}{r^2}, 1)\to \Bern(\Phi(P_1))$, and $\N(\frac{(r^2-1)\gamma}{r^2}, 1)\to \Bern(\Phi(P_2))$, so the strength matches (and hence the case for $H_0$ is proven).

\spac Note that in our previous step in $H_1$, in each block $S_i\times S_j$, a sub-block $A_{t_i}\times A_{t_j}$ is elevated to $\Bern(P_2)$ whereas the rest are $\Bern(P_1)$. This means that in the overall graph, the sets:
$$\bigcup_i \pi_i^{-1}(A_{t_i})\times \bigcup_i \pi_i^{-1}(A_{t_i})$$have elevated density $\Bern(P_2)$ where the rest has density $\Bern(P_1)$. Finally, note that the total size of $\bigcup_i \pi_i^{-1}(A_{t_i})$ is exactly $\sum_{i=1}^{k_0}\frac{n}{k_0r}=\frac{n}{r}$. Therefore, after permuting the nodes we get exactly $\pd(n, n/r, P_2, P_1)$ as the output.
\end{proof}
\begin{proof}[\sref{Theorem}{PDSreduction}]
Define the steps of $\mathcal{A}$ to map inputs to outputs as follows
$$(G, E) \xrightarrow{\mathcal{A}_1, \epsilon_1} (F, S) \xrightarrow{\mathcal{A}_2, \epsilon_2} M \xrightarrow{\mathcal{A}_{\text{3}}, \epsilon_3=0} G'$$
where the following $\epsilon_i$ denotes the total variation difference in each step (from output of $\mathcal{A}$ to the next target). Under $H_1$, consider the following sequence of distributions:
\allowdisplaybreaks
\begin{align*}
\mathcal{P}_0 &= G_E(N, k, p, q) \\
\mathcal{P}_1 &= \mathcal{M}_{[n] \times [n]}(S \times S, \textnormal{Bern}(p), \textnormal{Bern}(Q)) \quad \text{where } S \sim \mathcal{U}_n(F) \\
\mathcal{P}_2 &= \gamma \cdot \one_S\otimes \one_S + \N(-\frac{\gamma}{r^2}, 1)^{\otimes n \times n} \quad \text{where } S \sim \textit{Unif}_n(k) \\
\mathcal{P}_{\text{4}} &= \pd(n, k, P_{1}, P_{2})
\end{align*}
Applying \sref{Lemma}{TOBIPARTITE} before, we can take
$$\epsilon_1 = 4k_0 \cdot \exp\left( - \frac{Q^2N^2}{48pk_0n} \right) + \sqrt{\frac{C_Q k_0^2}{2n}}$$
where $C_Q = \max\left\{ \frac{Q}{1 - Q}, \frac{1 - Q}{Q} \right\}$. For $\epsilon_2$, \sref{Lemma}{BernRotation} guarantees that $\epsilon_2=O(n^{-1})$ suffices. The final step $\mathcal{A}_3$ is exact and we can take $\epsilon_3 = 0$. Finally, note that from the data processing inequality applied to $\dTV$ that $\dTV(\mathcal{A}_i(\cdot), \mathcal{A}_i(\cdot '))\leq \dTV(\cdot, \cdot')$ so each step the total variation loss at most accumulates (\sref{Lemma}{TVACC}), thus by the triangle inequality on TV we get $$\dTV(\mathcal{A}(G_{E}(N, k, p, q)),\pd(n, k, P_{1}, P_{2}))\leq \epsilon_1+\epsilon_2=o(1).$$
Under $H_0$, consider the distributions
\allowdisplaybreaks
\begin{align*}
\mathcal{P}_0 &= G(N, q) \\
\mathcal{P}_1 &= \text{Bern}(Q)^{\otimes n \times n} \\
\mathcal{P}_3 &= \N(0, 1)^{\otimes n \times n} \\
\mathcal{P}_{\text{4}} &= G(n, 1/2)
\end{align*}
As above, Lemmas \sref{Lemma}{TOBIPARTITE}, \sref{Lemma}{BernRotation} and \sref{Lemma}{Gaussianization} imply that we can take
$$\epsilon_1 = 4k_0 \cdot \exp\left( - \frac{Q^2N^2}{48pk_0n} \right), \quad \epsilon_2 = O(n^{-1}), \quad \text{and} \quad \epsilon_{\text{3}} = 0$$
Again by the data processing inequality (\sref{Lemma}{TVACC}), we therefore have that
$$\dTV\left( \mathcal{A}\left( G(N, q) \right), G(n, 1/2) \right) = O(\epsilon_1+\epsilon_2)=o(1)$$
which completes the proof of the theorem.\end{proof}
\subsection{Proof of \sref{Theorem}{PDSbound}}
Suppose we now have a direct reduction to $\pd(n, k, P_1, P_2)$ following the previous notations,
the final step of reduction have to do with the making the uniform degree condition exact, and applying it to general (dense) $P_0$. Consider the following post-reduction process with given target $(P_0, p_1, p_2)$ such that $P_0=p_1-(\frac{n^2}{k^2}-1)\delta=p_2+\delta$ for some $\delta$:
\begin{enumerate}
    \item Apply $k$-PDS-to-$\pds$ on given instance $G_E(n, k, p, q)$ and output $G_1$ with specified $\mu, \gamma$ such that the exact condition $\Phi(\frac{(r^2-1)\gamma}{r^2})=\frac{P_1}{2P_0}$ holds. As before, denote output density $P_1, P_2$ (then $p_1=2P_0P_1$, and $\delta$ can be expressed by $P_0, \gamma, r$).
    \item If $P_0>1/2$, then output $G_2$ by including all edges in $G_1$ and independently including all non-edge in $G_1$ with probability $2P_0-1$, else include all edges in $G_1$ with probability $2P_0$. 
\end{enumerate}
We show that the output of the above second post-processing step $\mathcal{A}_4$ satisfies:
$$\mathcal{A}_4(G(n, 1/2))=G(n, P_0)$$
$$\dTV(\mathcal{A}_4(\pd(n, k, P_1, P_2)),\, \pd(n, k, p_1, p_2))=o(1)$$
These two equations completely settles the reduction from PC to $\pds$ to the general density.

\spac Note that the first equation concerning $H_0$ is trivial, because a $\Bern(s)$ instance get transferred (independently) directly to $\Bern(2P_0s)$ by $\mathcal{A}_4$. Thus we only need to deal with the second equation. The general insights is that, when $\nu$ is small, $\Phi(\nu)$ (Gaussian CDF) is almost a linear function of $\nu$ where $\Phi(\nu)\sim \frac{1}{2}+\frac{1}{\sqrt{2\pi}}\nu$ and the error term (when $\nu<0.1$) is: $$\left|\Phi(\nu)-\frac{1}{2}-\frac{1}{\sqrt{2\pi}}\nu\right|=\left|\frac{1}{\sqrt{2\pi}}\int_0^{\nu}(e^{-x^2/2}-1)dx\right|\leq \frac{1}{\sqrt{2\pi}}\left|\int_0^{\nu} x^2dx\right|=\frac{1}{3\sqrt{2\pi}}|\nu|^3$$since $|e^x-1|<2|x|$ when $|x|<0.01$. Therefore the average degree condition \emph{approximately} but not exactly holds with $P_1$ and $P_2$ already.

\spac Formally, note that $\mathcal{A}_4(\pd(n, k, P_1, P_2))=\pd(n, k, p_1, 2P_2P_0)$, and we only need to show that $\dTV(\pd(n, k, p_1, 2P_2P_0),\pd(n, k, p_1, p_2))=o(1)$. The trick here is to use the data processing inequality again: because the distribution $\pd(n, k, p, q)$ is obtained by applying the (random) planted dense subgraph over $G(n, q)$, thus the total variation:
\begin{align*}
    \dTV(\pd(n, k, p_1, 2P_2P_0),\pd(n, k, p_1, p_2))&\leq \dTV(G(n, 2P_2P_0), G(n, p_2))\\
    &=\dTV(\Bern(2P_0P_2)^{\otimes {n\choose 2}}, \Bern(p_2)^{\otimes {n\choose 2}}).
\end{align*}
Moreover, by \sref{Lemma}{TVBIN} we know that the above is bounded by $|2P_0P_2-p_2|\cdot O(n)$ because the denominator $P_0(1-P_0)\in \Theta(1)$. Now we only need to prove that $|2P_0P_2-p_2|\in o(n^{-1})$. Note that this can be computed as exactly:
\begin{align*}
    |2P_0P_2-p_2|=&2P_0\left|\Phi(\frac{-\gamma^2}{r^2})-\frac{1}{2}+\frac{1}{r^2-1}\left(\frac{p_1}{2P_0}-\frac{1}{2}\right)\right|\\
    =&2P_0\left|\Phi(\frac{-\gamma^2}{r^2})-\frac{1}{2}+\frac{1}{r^2-1}(\Phi(\frac{(r^2-1)\gamma}{r^2})-\frac{1}{2})\right|\\
    \leq& 4\frac{\gamma^3}{r^2}=\frac{\mu}{n}\left(\frac{k_0^2}{n}\right)\left(\frac{k_0r}{n}\right)^{2.5}=o(n^{-1})
\end{align*}because $\frac{k_0^2}{n}<\frac{k_0r}{N}, \frac{k_0r}{n}<\frac{k_0^2}{N}$ are all assumed to be smaller than one, and $\mu\to_n 0$.

\spac We now turn to the formal lower bound from the reduction. Consider the following parametrized model $\pd(n, k, p_1, p_2)$ versus $G(n, p_0)$ such that:
$$p_0=p_1-\frac{r^2-1}{r^2}\gamma=p_2+\frac{1}{r^2}\gamma$$
We prove that there is a computational threshold for all signal levels below $\gamma^2\in\tilde o((\frac{r^{2}}{n})^{1.5})$ by filling out all possible growth rates below.

\spac Note that fix $P_0\in(0,1)$ throughout (we only need it to be bounded away from 0 and 1), for the reduction to work with a given sequence $(N, k_0, p, q)$ to $n=kr$ where $r\in\tilde o(k)$, $k_0\in\tilde o(n^{1/2})$ and $k\in\tilde \omega(k_0)$ are (implicit) functions of $N$, we only need to characterize the range of viable signal strength $\gamma$ that can be reduced to:
$$\gamma=\mu(\frac{k_0r}{n})^{1.5}>\frac{1}{w(n)\sqrt{\log n}}(\frac{k_0r}{n})^{1.5}$$ asymptotically where $w$ can be any (slowly) increasing unbounded function (such as $n^{o(1)}$). Note that this range do indeed cover the entirety of $\tilde o((\frac{r^{2}}{n})^{1.5})$ assuming $k_0\in \tilde\Theta(N^{0.5})$.

\spac Therefore, we know that by the PC conjecture and the given reduction the computational lower bound for $\pds$ holds up to the upper bound level in \sref{Proposition}{UpPDS}. 
\subsection{Proof of \sref{Corollary}{RECOVBOUND}}
\begin{proof}
By \sref{Lemma}{recovtodetect}, we only need to show that a weak recovery blackbox output is a qualifying secret key $k(G)$ for refutation.

\spac Consider a set $R$ that overlaps with the real $\pd$ planted set with size $\rho>1/2$ ($\rho\to_n 1$ holds for weak recovery, but for the sake here we only need it at least $1/2$ for convenience). Consider $\pd$ density parameters $p>q>n^{-1}\log n,\, p\in O(q)$, and consider the sequence of $r_n=\frac{p_n+q_n}{2}$ such that $p=O(q)=O(r)$ and $D_{KL}(p\|r)=\Theta(D_{KL}(p\| q))\subset\tilde\omega(k^{-1})$. However, by flipping the graph for \sref{Theorem}{Thm: StatDKS} in the dense case $\lim p=\lim q=p_0$, we know that the smallest $\rho k$-subgraph in $G(n, p)$ has density at least $r$ with high probability. Thus the density of $R$ is at least $\frac{1}{4}(r+3q)=\frac{7q+p}{8}:=s$ with high probability.

\spac However, by \sref{Theorem}{Thm: StatDKS} again we note that the densest $k$ subgraph in $G(n, p_0)$ will not be of density at least $\frac{s+p_0}{2}$ because $(s-p_0)=\Theta(p-q)$ so $D_{KL}((s+p_0)/2\| p_0)=\Theta(D_{KL}(p\| q))\subset\tilde\omega(k^{-1})$. Therefore with high probability the densest $k$ subgraph of $G(n, p_0)$ has a gap with the density of $1/2$ portion recovered densest $k$ subgraph in $\pds$. By \sref{Lemma}{recovtodetect} and \sref{Theorem}{PDSbound} we are done with the proof.
\end{proof}
\subsection{Statistical Optimal Boundary for $\pds$}
 It is well known that the success of a statistical hypothesis testing between two distributions $P, Q$ from one sample depends on $\dTV(P, Q)$. However, because our alternate hypothesis is composite (mixture over latent $\theta$), it can be challenging to compute the total variation distance between mixture $\mathbb{E}_{\theta} P_\theta$ and null $P_0$ beyond trivial geometric bounds. Thus alternative methods are needed.
 
 \spac In this section, for the completeness of our results on $\pds$, we also present a statement of the statistical boundaries drawing comparisons with a line of statistical lower bounds in the canonical PDS such as \cite{Ingster13,HWX15,MaWu} where their upper bound construction with mean comparison is now invalid in $\pds$. While we can derive asymptotically similar lower bounds, there is provably no polynomial test matching this boundary in $\pds$ \emph{and} recovery is impossible.  
Instead, we derive a boundary necessary for the $\chi^2$ divergence between $H_0$ and $H_1$ to be large via the Ingster's trick to handle mixture in the latent structure.
\begin{theorem} [Statistical lower bounds for $\pds$]\label{StatBound} Consider $\pds$ when $0<q<p_0<q<1/2$. Consider the setting with a sequence of edge densities $p^{(n)}, p_0^{(n)}, q^{(n)}, k^{(n)}$ with graph size $n\to\infty$:
\begin{itemize}
    \item If $\lim\sup\frac{k^{(n)^4}}{n^2}\cdot\frac{(p^{(n)}-q^{(n)})^2}{q^{(n)}(1-q^{(n)})}\to0$ and $\lim\sup k^{(n)}\frac{(p^{(n)}-q^{(n)})^2}{q^{(n)}(1-q^{(n)})}\to 0$, then no (statistical) test on $\pds$ on those parameters can achieve type I + type II error strictly less than $1$ asymptotically.
\end{itemize}

\end{theorem}
\begin{proof}
Consider the $\chi^2$ trick applied on the mixture: $P_0=G(n, p_0)$ and $P_{\theta}=\pd(n,\theta, p, q)$ with planted set at $\theta$ and $\theta\sim{n\choose k}$ be uniformly distributed.
\begin{align*}
    \chi^2(\mathbb{E}_{\theta}(P_{\theta})\|P_0)&=\int_G \frac{\mathbb{E}_{\theta}(P_{\theta}(G))\mathbb{E}_{\theta'}(P_{\theta'}(G))}{P_0(G)}-1\\
    &=\mathbb{E}_{\theta\indep\theta'}\frac{P_{\theta}(G)\mathbb{E}_{\theta'}(G)}{P_0(G)}-1
\end{align*}
If we expand the above expression and denote the $\lambda=\chi^2(p, q)=\frac{(p-q)^2}{q(1-q)}$ then we end up with the above (tightly) upper bounded by:
$$E(\exp (\lambda (H^2-E(H)^2))$$where $H\sim \theta\bigcap\theta'$ is distributed according to Hypergeometric$(n, k, k)$ (where $k>\sqrt{n}$). This evaluation goes to zero from a local inequality on Hypergeometric inequalities (Lemma 6 in Appendix C of \cite{HWX15}), which concludes our proof.

\spac A better way to view it (from reductions) is as follows: we know that when the inequality condition holds, $\pd(n, k, p, q)$ is in-distinguishable from $G(n, k, q)$ by the PDS boundary in \cite{BBH18}. However, in this case we have:
\begin{align*}
    \dTV(G(n, q), G(n, p_0))&\leq \dTV(\Bern(q)^{\otimes n^2},\Bern(p_0)^{\otimes n^2})\\
    &\leq n\sqrt{\frac{(q-p_0)^2}{q(1-q)}}=\frac{n}{r^2}\sqrt{\lambda}
\end{align*}
by \sref{Lemma}{TVBIN}. Therefore, $\dTV(G(n, q), G(n, p_0))\to 0$ below the statistical $\pd_D$ testing threshold, meaning that $\pds_D$ can also not be performed. This reduction proves the intuition that $\pds_D$ is ``harder" than $\pd$.\end{proof}
\begin{remark}
We also remark that the reverse direction for the above lower bound is still open. Let $H$ be distributed according to Hypergeo$(n, k, k)$ (where $k>\sqrt{n}$), then $E(H)=k^2/n$. Assuming that $\lambda E(H)^2\in \tilde\omega(1)$, if one can prove that:
$$E(\exp (\lambda (H^2-E(H)^2))\to\infty$$as well, which is a stronger version of Lemma 6 in \cite{HWX15}, then $\chi^2$ between two hypotheses of $\pds_D$ diverges, which is still a necessary (insufficient) condition for the upper bound. It is of interest to show a tight statistical detection upper bound for models like $\isbm$ and $\pds$ at the regime when it is computationally infeasible and statistically impossible to recover.
\end{remark}

\section{Proofs for Refutation}
\subsection{Proof of \sref{Theorem}{ISBMBound}}
\begin{theorem}[Reduction from $k-\pd$ to $\isbm$]\label{PDSB}
Let $N, k_0$ be parameters of planted clique graph size, $r<N/k_0$ be a target output for ratio of planted set. Let $\alpha$ be a constant and assume $r<\left(\frac{N}{k_0}\right)^{1-\alpha}$. We present the following reduction $\phi$ with absolute constant $C_{\alpha}>1$:
\begin{itemize}
    \item \textup{Initial $k$-$\pd$ Parameters:} vertex count $N$, subgraph size $k_0\in o_N(N)$ dividing $N$, edge probabilities $0 < q < p \leq 1$ with $\min\{q, 1 - q, p-q\} = \Omega(1)$, and a partition $E$ of $[N]$. We further assume that $k_0\in o(\sqrt{N})$ holds.

    \item \textup{Target $\pd$ parameters:} $(n, r, k)$ where $r\in o_n(\sqrt{n})$ is the specified parameter and $k$ is the expected subgraph size $k=n/r$ and $n$ is the smallest multiple of $k_0r$ that is greater than $(1+\frac{p}{Q})N$ where $$Q=1-\sqrt{(1-p)(1-q)}+1_{p=1}(\sqrt{q}-1).$$
    
    \item \textup{Target $\isbm$ edge strength:}
    $$\gamma=\mu(\frac{k_0r}{n}),\;\;\;P_{11}=\Phi(\frac{(r-1)^2\gamma}{r^2}) ,\;\;\; P_{12}=\Phi(-\frac{(r-1)\gamma}{r^2}),\;\;\;P_{22}=\Phi(\frac{\gamma}{r^2})$$
    where $\mu\in(0, 1)$ satisfies that$$\mu\leq\frac{1}{12C\sqrt{\log(N)+\log(p-Q)^{-1}}}\cdot\min\{\log(\frac{p}{Q}), \log(\frac{1-Q}{1-p})\}.$$
    \item Applying $\phi$ on the given input graph instance $G$  yields the following (when $k_0\in o(\sqrt{N})$):
$$\dTV(\phi(G(N,\frac{1}{2})), G(n, \frac{1}{2}))=o_n(1)$$
$$\dTV(\phi(PC_{\rho}(N, k_0, \frac{1}{2})), \isbm(n, k, P_{11}, P_{12}, 
P_{22}))=o_n(1)$$
\end{itemize}
\end{theorem}
 \begin{figure}[h!]
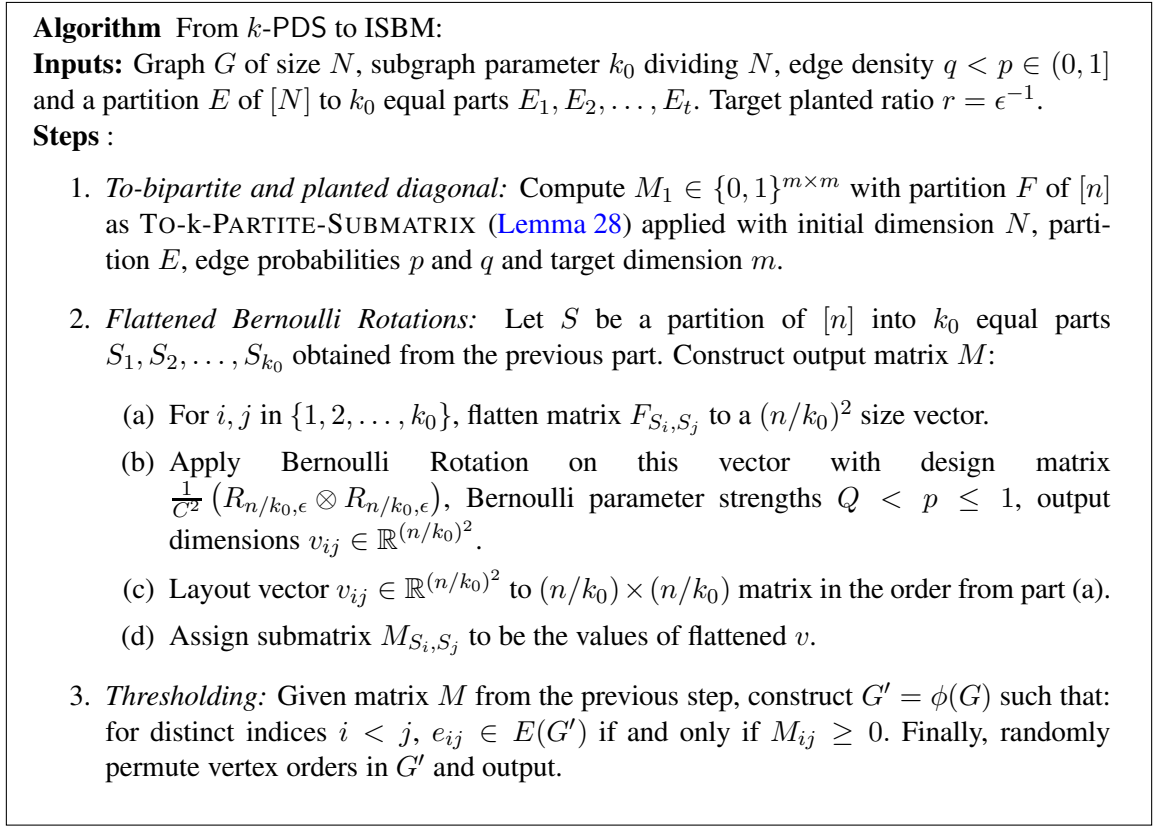

     \begin{mdframed}[innerrightmargin=15pt]
 \textbf{Algorithm } From $k$-$\pd$ to $\isbm$:
 
\spac        \textbf{Inputs:} Graph $G$ of size $N$, subgraph parameter $k_0$ dividing $N$, edge density $q < p\in (0, 1]$ and a partition $E$ of $[N]$ to $k_0$ equal parts $E_1, E_2,\dots, E_t$. Target planted ratio $r=\epsilon^{-1}$.
         
\spac        \textbf{Steps }:
        \vspace{-3pt}
        \begin{enumerate}
             \item \textit{To-bipartite and planted diagonal:} Compute $M_1 \in \{0, 1\}^{m\times m}$ with partition $F$ of $[n]$ as
             \textsc{To-}k\textsc{-Partite-Submatrix} (\sref{Lemma}{TOBIPARTITE})
            applied with initial dimension $N$, partition $E$, edge probabilities $p$ and $q$ and target dimension $m$.
            \item \textit{Flattened Bernoulli Rotations:} Let $S$ be a partition of $[n]$ into $k_0$ equal parts $S_1, S_2,\dots, S_{k_0}$ obtained from the previous part. Construct output matrix $M$:
            \begin{enumerate}
                \item For $i, j$ in $\{1,2,\dots,k_0\}$, flatten matrix $F_{S_i, S_j}$ to a $(n/k_0)^2$ size vector.
                \item Apply Bernoulli Rotation on this vector with design matrix $\frac{1}{C^2}\left( R_{n/k_0,  \epsilon}\otimes R_{n/k_0, \epsilon}\right)$, Bernoulli parameter strengths $Q<p\leq 1$, output dimensions $v_{ij}\in \mathbb{R}^{(n/k_0)^2}$.
                \item Layout vector $v_{ij}\in \mathbb{R}^{(n/k_0)^2}$ to $(n/k_0)\times (n/k_0)$ matrix in the order from part (a).
                \item Assign submatrix $M_{S_i, S_j}$ to be the values of flattened $v$.
            \end{enumerate}
            \item \textit{Thresholding:} Given matrix $M$ from the previous step, construct $G'=\phi(G)$ such that: for distinct indices $i<j$, $e_{ij}\in E(G')$ if and only if $M_{ij}\geq 0.$ Finally, randomly permute vertex orders in $G'$ and output.
        \end{enumerate}
        \vspace{.5mm}
        \end{mdframed}
     \caption{\footnotesize Reduction from $k$-PDS to $\isbm$.}
 \end{figure}
The key distinction here from the reduction in $\pds$ lies almost solely in the design matrix, which is simply the Kronecker (tensor) product of two matrices given by \sref{Lemma}{RegularMatrix}. Therefore:
\begin{lemma}[Bernoulli Rotation for $\isbm_D$]
Consider the second step $\mathcal{A}_2$ applied on the output of \sref{Lemma}{TOBIPARTITE} and assuming notations through \sref{Lemma}{BernRotation}, we have:
\begin{align*}
    \dTV(\mathcal{A}_2(\Bern(Q)^{\otimes n\times n}, \N(0,1)^{\otimes n\times n})&=o(n^{-1})\\
    \dTV(\mathcal{A}_2(\mathcal{M}_{[n]\times [n]}(\mathcal{U}_n(S), p, Q)),\; \mathcal{L}( \frac{\gamma}{r^2}\cdot (rv-{1})^T(rv-{1})+\mathcal{N}(0, 1)^{\otimes n\times n}))&=o(n^{-1})
\end{align*}
where $v\sim \textit{Unif}_n(k)$.
\end{lemma}

\begin{proof}
Similar to \sref{Lemma}{BernRotation}, the only different part is the output to Bernoulli Rotation in $H_1$ after performing distribution shifts to $\N(0, 1)$ and $\N(\mu, 1)$.

\spac We analyze the output for the specific design matrix. For the sub-block $S_i\times S_j$, it gets mapped to exactly the flattened product of the $t_i$th row of $\frac{1}{C}R_{n, r^{-1}}$ (transposed) times the $t_j$th row of $\frac{1}{C}R_{n, r^{-1}}$. Denote the set of positive terms in the $t$th row to be $P_t$, then the output distribution (conditioning on $T$, the source planted set) is exactly:
$$\mathcal{L}(\frac{\gamma}{r^2}\cdot (rv-{1})^T(rv-{1})+\mathcal{N}(0, 1)^{\otimes n\times n}))$$for $v=\prod_{i=1}^{k_0} P_{t_i}$ (this can also be viewed in the light \emph{Tensor Bernoulli Rotation}, see Corollary 8.2 in \cite{BB20}). Afterwards we can simply permute the nodes (note that $v$ has size $k=n/r$) and the Bernoulli rotation target follows.

\spac To finish off, we get the exact same bound on total variation at most $o(n^4R_{rk}^{-3})$, at which point the total variation bound from applying this algorithm step by $o(n^{-1})$.
\end{proof}
After the Bernoulli rotation, we proceed in a similar fashion with \sref{Theorem}{PDSreduction}. We defer a formal full proof in thresholding Gaussians and aligning the precise density to Corollary 14.5 and Theorem 3.2 in \cite{BB20}, here we only remove the condition (\textsc{T}) they imposed during Bernoulli Rotation to imply the lower bound results in a general regime. This gives us the desired lower bound.
\subsection{Proof of \sref{Theorem}{REFUTEBound}}
\begin{proof}[Refutation hardness]
Consider applying \sref{Lemma}{refutetodetect}, we know that as long as we can find a satisfying ``quiet" adversarial distribution $H_1$, such that\begin{itemize}
    \item $H_1$ is computationally indistinguishable from $H_0$.
    \item $H_0, H_1$ satisfies the refutation valuation function criteria.
\end{itemize}
then we can claim that refutation is hard for $H_0$. For the case of the Erd\H{o}s-Renyi graph null hypothesis $H_0: G(n, P_0)$ and $\val=\DkS$, we may simply consider the alternate hypothesis as $H_1: \isbm(n, k, P_{11}, P_{12}, P_{22})$ with specific pairs of parameters.

\spac
In fact, when $k(P_{11}-P_0)\in\tilde\omega(1)$ we know that the densest subgraph in $H_1$ has density at least $\frac{2P_{11}+P_0}{3}$ with probability $1-o_n(1)$ from the Markov inequality. And by \sref{Theorem}{Thm: StatDKS} we know that as long as $k(P_{11}-P_0)^2\in\tilde\omega(1)$ the (statistical) densest $k$-subgraph in $H_0$ is smaller than $(P_0+P_{11})/2$ with high probability. These two constraints of parameter growth will be satisfied because the (optimal) output regime for \sref{Theorem}{ISBMBound} actually reads $(P_{11}-P_0)^2\in\tilde \Theta(n/k^2)$.

\spac Therefore, from \sref{Theorem}{ISBMBound} we know that as long as $H_0$ is indistinguishable with $H_1$, one cannot refute in polynomial time the densest $k$-subgraph in $G(n, P_0)$ to have value larger than $\frac{P_0+2P_{11}}{3}:=q$. Plugging in the boundary for $\isbm_D$ we know that refutation (of $\pd$) is computationally impossible under the regime $$\frac{k^2D_{KL}(p\|q)}{n}\in\tilde o(1)$$ which contrasts the detection threshold $\frac{k^4D_{KL}(p\|q)}{n^2}\in O(1)$ above which one can perform the optimal sum-test. This fact, combined with semi-definite programming, completely resolves the refutation problem of $\DkS$ in Erd\H{o}s-Renyi graphs.
\end{proof}

\subsection{Computational upper bound for refutation}
To prove an upper bound for refutation, we first need to introduce the semi-definite programming relaxation, which is a common method to computational approach problems such as densest-$k$-subgraph, considered in many works such as \cite{HWX16, CX14}.

Consider the following relaxation of the densest-subgraph:
\begin{align}
\widehat{Z}_{SDP} = \arg\max_{Z}  & \langle E, Z \rangle     \nonumber     \\
\text{s.t.	} & Z \succeq 0,\quad Z\geq 0   \label{eq:PDSCVX_SL}  \\
& Z_{ii} \leq 1, \quad \forall i \in [n]\nonumber \\
& \langle I, Z \rangle  = k  \nonumber  \\
 & \langle J, Z \rangle = k^2.  \nonumber
\end{align}
It is not hard to see that:
\begin{itemize}
    \item A feasible solution of a true subgraph is also feasible for (\sref{}{eq:PDSCVX_SL}), thus the latter will always return objective at least the true density.
    \item (\sref{}{eq:PDSCVX_SL}) is a semi-definite programming problem, and can be efficiently solved.
\end{itemize}
With the sufficiency results given in \cite{HWX16} (specifically, combine their results in Lemma 14, Lemma 15, and Theorem 5), we can show that under the separation conditions of $\frac{k^2}{n}\frac{(p-q)^2}{q(1-q)}\to\infty$ and $k\frac{(p-q)^2}{p(1-p)}\to\infty$, the above formulation of convexified programming for planted dense subgraph will have the optimal solution converging to the true planted instance of our graph $P(\hat Z_{SDP}=Z)\to 1$ assuming the null density satisfies $0.9>q\in\Omega(\frac{\log n}{n})$. Here we use their results for the objective function instead (that is, the objective $\langle E, Z \rangle\leq k^2p+ck\sqrt{p(1-p)}$ with probability $\to 1$ as constant $c\in\Omega_{n, k}(1)$ by Markov Inequality).

\begin{theorem}
Consider the semi-definite programming relaxation for $G\sim G(n, q)$. Then when $\frac{k^2}{n}\frac{(p-q)^2}{q(1-q)}\to\infty$ and $k\frac{(p-q)^2}{p(1-p)}\to\infty$, the probability that the objective function is at least $k^2 \frac{q+2p}{3}$ goes to zero as $n, k\to\infty$. Moreover, in $H_1=\pd(n, k, p, q)$, the objective will be at least $\frac{q+4p}{5}$ with probability $\to 1$. This means that (\sref{}{eq:PDSCVX_SL}) will successfully refute the densest $k$-subgraph valuation problem in $G(n, q)$ vs $\pd(n, k, p, q)$.
\end{theorem}

\begin{proof}
We start with the following lemma from stochastic domination: 
\begin{lemma}
Let $F(P)$ be the distribution of objective (\sref{}{eq:PDSCVX_SL}) under the graph distribution $P$. If edges in $G\sim P$ are sampled independently with probability matrix $E_P$ for two distributions $P, Q$, such that $E_P-E_Q\geq 0$ (entry-wise), then for any $x>0$, $\mathbb{P}(F(P)>x)\geq \mathbb{P}(F(Q)>x)$. In other words, the convex program is monotone with respect to the underlying density.
\end{lemma}
\begin{proof}
Consider the following process:
\begin{enumerate}
    \item On $G\sim Q$, find optimal $\hat Z$ for  (\sref{}{eq:PDSCVX_SL}).
    \item Update $G$ in the following way: for any $e_G=0$, flip $e_G=1$ with probability $\frac{E_p(e)-E_q(e)}{1-E_q(e)}$.
\end{enumerate}
The objective never decreases because we only add edges in the second step, whereas the unconditional distribution of the graph generated from 2 is exactly $P$. Hence we find a coupling between two distributions of graphs such that $F(P|G)$ is bounded below by $F(\{G\})$ for any $G$, and the result follows.
\end{proof}
Moreover, note that the above lemma applies to the mixture problem too. Since (\sref{}{eq:PDSCVX_SL}) is symmetric, the objective will not change if we condition the planted dense subgraph to a specific location, then we can use the above lemma and conclude that $F(\pd(n, k, p, q))$ dominates $F(G(n, q))$ for any density $p>q$. Consider the alternative $\pd(n, k, (p+q)/2, q)$ for (\sref{}{eq:PDSCVX_SL}), which also satisfies the conditions for successful recovery by (\sref{}{eq:PDSCVX_SL}). Note that in this case, we know that the objective:
$$\frac{1}{k^2}\langle E, \hat Z \rangle\leq \frac{p+q}{2}+\frac{c}{k}\sqrt{p(1-p)}$$ with probability $\to 1$ if $c\to\infty$. Consider plugging in $p$ to the RHS we get $c=\frac{k}{6}\cdot\frac{(p-q)}{\sqrt{p(1-p)}}\to\infty$ by the asymptotic conditions. Thus that the above objective is bounded above by $k^2\cdot \frac{q+2p}{3}$ with probability going to 1.

\spac On the other hand, clearly for $X\sim\Binom(k^2, p)$, we have $X\leq k^2\cdot\frac{4p+q}{5}$ with probability at most $$\mathbb{P}(X\leq k^2\cdot\frac{4p+q}{5}| \Binom(k^2, p))\leq (\frac{k(p-q)}{5\sqrt{p(1-p)}})^{-2}\to 0$$by Markov inequality. So the valuation condition for $H_1$ is met.
\end{proof}
\subsection{Statistical bounds for refutation}
We now turn to show that the statistical limit of DkS problem lies upon recovery boundary for $G(n, q)$ (ignoring log factors). This has also been studied under the name quasi-cliques ($k$-subgraphs with edge count at least $\gamma{k\choose 2}$) in random graphs by a line of works such as \cite{VBKJ,AS16,DkSErdos}. However, our regime of interest ($k=\Theta(n^{\alpha}), q\in\Omega(1)$) remains largely unstudied in past literature.
\begin{theorem}\label{Thm: StatDKS}
Consider $d=D_{KL}(\Bern(p)\|\Bern(q))=\Theta(\frac{(p-q)^2}{q(1-q)})$ when the densities $p/q\to \Theta(1)$ and $np>nq>\log n$. Then the densest $k$ subgraph density of $G(n, q)$ will be smaller than $\frac{p+q}{2}$ with probability $\to 1$ if $\frac{kd}{\log n}\to\infty$ and $k\to\infty$. Thus statistical refutation is possible.
\end{theorem}
\begin{proof}
Firstly we need a tail bound on the Binomial distribution (for $r:=\lceil{pN}\rceil$):
\begin{align*}
    \mathbb{P}(\Binom(N, q)\geq pN)&\leq N\cdot\mathbb{P}(\Binom(N, q)=r)\\
    &=N{N\choose{r}}q^r(1-q)^{N-r}\\
    &\leq N^2\frac{N!}{r!(N-r)!}e^{N(p\log q+(1-p)\log (1-q))}\\&<2N^2\frac{1}{\sqrt{2\pi p(1-p) N}}e^{N D_{KL}(p\|q)}=e^{-N D_{KL}(p\|q)+O(\log N)}
\end{align*}from Stirling's formula and $N, r\to \infty$.

\spac Therefore, we can go on to look at each block, which has $k^2$ independent Bernoullis and thus satisfies the density tail with probability at most $e^{-k^2d+O(\log k)}$. However, there are at most $n\choose k$ such blocks, so if assign random variables $X=\sum X_i$ to those we have:
$$\mathbb{P}(X>0)\leq \mathbb{E}(X)=\sum \mathbb{E}(X_i)\leq n^ke^{-k^2d+O(\log k)}=e^{-k^2d+k\log n+O(\log k)}$$when $\frac{kd}{\log n}\to\infty$, we know that the above objective goes to zero. Replacing $p$ with $\frac{p+q}{2}$ for the above arguments works the same, and thus we are done. 

\spac As an extension, when $\frac{k d_k}{\log n}\to_k\infty$ with parameters $p_k$, we can show that (via a union bound) the densest $k$ subgraph density does not exceed $\frac{q+p_k}{2}$ for all $k>\log n$ simultaneously because the objective is bounded by $\exp\left(-k(kd-\log n)+O(\log k)\right)<\exp(-k\log n)=n^{-k}$).
\end{proof}

Next, we deal with the lower bound on refutation, which states that in $G(n, q)$ there is a dense subgraph with density $p$ and size $k$ with high probability if $kd\in\tilde o(1)$. The following theorem is sufficient to close the boundary for statistical impossibility. Assuming the same set of parameters, we have the following lower bound:
\begin{theorem}[Lower bound on refutation]
Assuming that $k\in o({n}^{\alpha})\bigcap \tilde{\omega}(\sqrt{n})$ for some fixed constant $\alpha<1$ and $p, q$ are all bounded away from 0 and 1\footnote{Observe that here the KL divergence reduces to $\Theta((p-q)^2)$ and $\log n/k=\Theta(\log n)=\Theta(\log k)$}. Moreover, for any $\alpha>0$, assume that $k(p-q)^2\in  \Theta(\log^{-1.01} n)$. There exist a $k$-subgraph in $G(n, q)$ with density at least $(p+q)/2$ with probability $1-o_n(1)$.
\end{theorem}
\begin{proof} 
First of all, the conditions assert that $p-q\in\Omega(\frac{\log^2 n}{n})$. In fact, from the statistical boundary on exact recovery given in \cite{HWX16b}, we know that recovery is impossible in this regime, even in the minimal variant (from the reduction given in \sref{Lemma}{partialtofull}).

\spac
Consider the densest $k$ subgraph estimator $\hat E$, which happens to be the MLE estimator on the planted instance (though we do not need this fact), we know that under this regime it correlates with the true planted mean with expected density $<\epsilon$ for any constant $\epsilon$ asymptotically (partial recovery), we try to bound the density in $\hat E$ before planting the dense subgraph. 

\spac 
Formally, assume that $\hat E\cap E=T$ where $E$ is the true planted set. Consider the original $G_1$ the instance from $G(n, q)$ and $G_1'$ be the graph after planting on $E$. The total edges in $G_1(\hat E)$ is at least (since it is the densest subgraph in $G_1'$):
$$E^{G_1}_{\hat E}= E^{G_1'}_{\hat E}-E^{G_1'}_{T}+E^{G_1}_{T}\geq E^{G_1'}_{E}-E^{G_1'}_{T}+E^{G_1}_{T}$$and we bound those terms one by one. To start, note that $|T|>\log n$, else the total edges offset in $T$ is at most ${|T|\choose 2}<(\log n)^2$, and ${|T|\choose 2}/{k\choose 2}=O((\log n)^2/k^2)\subset o(p-q)$. Now we consider the case when $|T|>\log n\to\infty$ and apply the densest $|T|$ subgraph in $(G_1')_E$:
    \begin{enumerate}
    \item $E^{G_1'}_{E}$ is just the edge count of the planted instance that is distributed according to $G(k, p)$. We know that the total number of edges is at least $$E^{G_1'}_{E}\geq \frac{2p+q}{3}{k\choose 2}$$from a simple Markov inequality (as in the previous theorem).
    \item $E^{G_1'}_{T}$ is equivalent to $|T|$-subgraph sampled from $G(n, q)$. From the previous theorem, we know that if $\frac{|T|(p-r_{|T|})^2}{\log k}\in \omega(1)$, then with probability $1-o(1)$ the densest $|T|$ subgraph in planted set has density at most $r_{|T|}$, and $E^{G_1'}_{T}<r_{|T|}{|T|\choose 2}$.
    
    \item Similar as the previous part, we know that when $\frac{|T|(p-s_{|T|})^2}{\log n}\in \omega(1)$ the probability that $G_1$ has such a \emph{sparse} subgraph is at most $1-o(1)$ (note that here we use the reverse side of the tail bound, which is a trivial implication when $p, q$ are bounded away by one) and $E^{G_1}_{T}>s_{|T|}{|T|\choose 2}$.
\end{enumerate}
Combining  the above, we only need to show that:
$$(r_{|T|}-s_{|T|}){|T|\choose 2}\leq \frac{1}{6}(p-q){k\choose 2}.$$

Let $d_{|T|}=r_{|T|}-p>0$ and $f_{|T|}=q-s_{|T|}>0$ then $|T|(d_{T}^2+f_{|T|}^2)\in O(\log n)$, thus the sum bound over all edges $|T|^2(d_{T}+f_{|T|})\in O(\sqrt{\log n}|T|^{3/2})$. Moreover, recall the condition on $p, q$ we have $k^2(p-q)\in\Theta(k^{3/2}\sqrt{\log^{-1.01}n}).$\footnote{Note that here the key is that (somewhat counterintuitively) we want $p, q$ to be far enough so that we can utilize the fact that small error terms cannot dominate the total density of at least $p$ in $\hat E$.}

\spac Now note that $r-s=p-q+(d+f)$, thus what remains to show is the local inequality $$|T|\in o(\frac{k}{\log^{0.7} k}),\;\;\;\;\;\;\text{when } k(p-q)\in\Theta(\log^{-1} n)$$after which apply the fact that ${|T|\choose 2}(d+f)\in o(k^2(p-q))$ we are done.

\spac Finally, note that the information theoretical limit for precise recovery is $k(p-q)^2\in\Theta(\log n)$ (\sref{Theorem}{PDSstatrec}), below which it is impossible to perform (even weak) recovery, so the above bound on $|T|$ follows immediately from \sref{Lemma}{partialtofull} with strength $\alpha=1.004$ (so the expected size of $|T|$ cannot be greater than $\frac{k}{(\log k)^{1.004}}<\frac{k}{(\log k)^{0.7}}$ by minimal recovery).
\end{proof}

As a conclusion to this section, we note that a version of the refutation bound was very recently closed in \cite{cheairi2022densest}, which states that the exact refutation boundary lies in $k(p-q)^2\in\Theta(\log n)$ (below which a dense subgraph exist with high probability). Though our theorem above is a weaker version of their result, the goal is to provide insights into reductions via recovery.

\section{Detection-Recovery gaps in other problems}\label{sec: othergap}
In this section, we finish our discussions on two other problems that observe a detection-recovery gap from reduction to a detection-recovery gap in PDS. In \cite{BBH18}, such relations were considered assuming the PDS recovery conjecture, here we do so at a lower rate of signal from only assuming PC conjecture and \sref{Theorem}{RECOVBOUND}. Denote the $H_1$ hypothesis distributions in \sref{Section}{OtherGaps} as $\textup{BC}(n, k, \mu)$ and $\textup{BSPCA}(m=n, k, d, \theta)$, respectively.
\subsection{Detection-Recovery gap in Biclustering}
This follows from a canonical process of simply performing $\textsc{To-}k\textsc{-Partite-Submatrix}$ and Gaussianizing. This gives us a symmetric planted Gaussian principal submatrix with elevated mean. Lastly, we can permute the columns if needed.

\begin{lemma}[Reduction to Bi-clustering -- Lemma 6.7 in \cite{BBH18}] \label{lem:bcrec}
Suppose that $n, \mu$ and $\rho \ge n^{-1}$ are such that
$$\mu = \frac{\log (1 + 2\rho)}{2 \sqrt{6 \log n + 2\log 2}}>\frac{\rho}{4 \sqrt{6 \log n + 2\log 2}}$$
Then there is a randomized polynomial time computable map $\phi = \textsc{BC-Recovery}$ with $\phi : G_n \to \mathbb{R}^{n \times n}$ such that for any subset $S \subseteq [n]$ with $|S| = k$, it holds that
$$\dTV\left( \phi\left(\pd(n, S, 1/2 + \rho, 1/2)\right), \mathbb{E}_{T\sim \textit{Unif}_n(k)} \mathcal{L}\left( \mu \cdot \one_S \one_T^\top + \N(0, 1)^{\otimes n \times n} \right) \right) = O\left(\frac{1}{\sqrt{\log n}} \right).$$
\end{lemma}
With this, we can now state the lower bound for recovery and refutation in BC by \sref{Lemma}{recovtodetect}:
\begin{corollary}[Recovery Hardness for Bi-Clustering]
Let $\alpha>0$ and $\beta\in(0,1)$, then there exists such parameters $(N_n, K_n, \mu_n)$ such that: (assuming the $\PCR$ detection hypothesis)
\begin{enumerate}
    \item The parameters are in the regime:
    $$\lim_{n\to\infty}\frac{\log K_n}{\log N_n}\leq \beta,\;\;\;\; \lim_{n\to\infty}\frac{\log \mu_n}{\log N_n}\leq -\alpha.$$
    \item If $\beta<\frac{1}{2}+\frac{2}{3}\alpha$, then there is no (randomized) polynomial-time recovery blackbox $\mathcal{A}_n: \mathbb{R}^{N_n\times N_n}\to {N_n\choose K_n}^2=:(\hat S, \hat T)$ such that $|\hat S\bigcap S|+|\hat T\bigcap T|-2K_n\in o(K_n)$ with probability greater than 0 asymptotically with $\mathcal{A}$ applied over the distribution on the Bi-clustering instance conditioning on $S, T$ and the uniform prior distribution $S\indep T\sim \textit{Unif}_n(k)$.
    \item If $\beta<\frac{1}{2}+\alpha$, then there is no polynomial-time refutation blackbox $\mathcal{A}_n: \mathbb{R}^{N_n\times N_n}\to \{0,1\}$ such that $\mathcal{A}$ returns 0 with asymptotically positive probability applied on $\N(0, 1)^{\otimes n\times n}$ and returns $\mathcal{A}(M)=1$ if there is a $k\times k$ submatrix $S$ in $M$ with mean at least $k^2\mu$. 
\end{enumerate}
\end{corollary}
We finally comment that the detection boundary is $\mu\in\tilde\omega(n/k^2)$ from the same reduction and hence the detection problem is computationally easy when $\beta>\frac{1}{2}+\frac{1}{2}\alpha$.
\subsection{Detection-Recovery gap in BSPCA}
\begin{theorem}[Recovery Hardness in BSPCA]
Let $\alpha \in \mathbb{R}$ and $\beta \in (0, 1)$. There exists a sequence $\{ (N_n, K_n, D_n, \theta_n) \}_{n \in \mathbb{N}}$ of parameters such that: (assuming the $\PCR$ detection hypothesis)
\begin{enumerate}
\item The parameters are in the regime
$$\lim_{n \to \infty} \frac{\log \theta_n}{\log N_n} \leq -\alpha, \;\;\;\; \lim_{n \to \infty} \frac{\log K_n}{\log N_n} \leq \beta$$
\item If $\alpha > \beta-\frac{1}{2}>0$, then there is no randomized polynomial-time recovery blackbox $\phi_n : \mathbb{R}^{D_n \times N_n} \to \binom{[N_n]}{k}^2$ such that the probability that $\phi_n$ recovers exactly the pair of latent row and column supports of an instance from $\textup{BSPCA}(N_n, K_n, D_n, \theta_n)$ is greater than 0 asymptotically, where the supports are independently distributed from the uniform prior.
\end{enumerate}
\end{theorem}
\begin{proof}
The proof follows from the following lemma:
\begin{lemma}[Random Rotation -- Lemma 8.7 in \cite{BBH18}] \label{lem:randrot}
Let $\tau : \mathbb{N} \to \mathbb{N}$ be an arbitrary function with $\tau(n) \to \infty$ as $n \to \infty$. There exists map $\phi : \mathbb{R}^{m \times n} \to \mathbb{R}^{m \times n}$ that sends $\phi(\N(0, 1)^{\otimes m \times n}) \sim \N(0, 1)^{\otimes m \times n}$ and for any unit vectors $u \in \mathbb{R}^m, v \in \mathbb{R}^n$ we have that
$$\dTV\left( \phi\left( \mu \cdot uv^\top + \N(0, 1)^{\otimes m \times n} \right),\, \N\left(0, I_m + \frac{\mu^2}{\tau n} \cdot uu^\top\right)^{\otimes n} \right) \le \frac{2(n + 3)}{\tau n - n - 3}\in o(1)$$
\end{lemma}
We defer the proof to \cite{BBH18} and focus on the reduction forward. Note that the left-hand side can be viewed as the asymmetric biclustering distribution, thus combining \sref{Lemma}{lem:bcrec} and \sref{Lemma}{lem:randrot} with \sref{Lemma}{TVACC} we get a polynomial-time map $\mathcal{A}$ such that:
$$\dTV(\mathcal{A}(\pd(n, u, \frac{1}{2}+\rho, \frac{1}{2})), \N(I_n+\frac{\mu^2}{\tau n}uu^T))=o(1).$$

Now the only thing left is to define precise parameter correspondence to apply \sref{Theorem}{RECOVBOUND}. Consider the following set of parameters (let $\gamma:=\beta-\frac{1-\alpha}{2}$):
$$K_n\in\tilde\Theta(N^{\beta}),\;\, \rho_n\in\tilde\Theta(N^{-\gamma}),\;\, N_n=D_n=N,\;\,\mu_n=\frac{\log (1 + 2\rho_n)}{2 \sqrt{6 \log N + 2\log 2}},\;\,\theta_n=\frac{k_n^2\mu_n^2}{\tau n}$$Observe that because $\rho\to 0$, $\log(1+2\rho)\in \Theta(\rho)$ and thus $\mu_n\in\tilde\Theta(\rho_n)$, one can easily verify that the conditions are equivalent to: $$\lim_{n\to\infty}\frac{\log(K_n^3\rho_n^2)}{\log(N_n)}=1-\alpha+\beta<1.5$$thus we can apply \sref{Theorem}{RECOVBOUND}, which concludes that no polynomial black-box can successfully recover the planted instance $u$ here. 
\end{proof}

\end{document}